\documentclass[10pt]{article}
\usepackage{amssymb}
\usepackage{amsmath}
\usepackage{fullpage}
\usepackage[T1]{fontenc}

\newcommand{\abs}[1]{\left| #1 \right|}

\newcommand{\w}[0]{\omega}
\newcommand{\N}[0]{\mathbb{N}}
\newcommand{\Z}[0]{\mathbb{Z}}

\newcommand{\scr}[1]{ \mathcal{ #1 } }
\long\def\symbolfootnote[#1]#2{\begingroup%
\def\thefootnote{\fnsymbol{footnote}}\footnote[#1]{#2}\endgroup} 
\newtheorem{theorem}{Theorem}[section]
\newtheorem{lemma}[theorem]{Lemma}
\newtheorem{proposition}[theorem]{Proposition}
\newtheorem{corollary}[theorem]{Corollary}
\newtheorem{claim}[theorem]{Claim}

\newenvironment{proof}[1][Proof]{\begin{trivlist}
\item[\hskip \labelsep {\bfseries #1}]}{\end{trivlist}}
\newenvironment{definition}[1][Definition]{\begin{trivlist}
\item[\hskip \labelsep {\bfseries #1}]}{\end{trivlist}}
\newenvironment{example}[1][Example]{\begin{trivlist}
\item[\hskip \labelsep {\bfseries #1}]}{\end{trivlist}}

\newcommand{\qed}{\hfill \mbox{\raggedright \rule{.07in}{.1in}}}

\title{Permutation Complexity and the Letter Doubling Map}
\author{Steven Widmer\footnote{(s.widmer1@gmail.com) The research presented here was supported by grant no. 090038012 from the Icelandic Research Fund.}}
\date{}
\begin{document}


%
%
%
%
	
\maketitle


\begin{abstract}
\noindent Given a countable set $X$ (usually taken to be $\N$ or $\Z$), an infinite permutation $\pi$ of $X$ is a linear ordering $\prec_\pi$ of $X$, introduced in $\cite{FlaFrid}$.  This paper investigates the combinatorial complexity of infinite permutations on $\N$ associated with the image of uniformly recurrent aperiodic binary words under the letter doubling map.  An upper bound for the complexity is found for general words, and a formula for the complexity is established for the Sturmian words and the Thue-Morse word.   
%
$\vspace{1.0ex}$

\noindent $\textbf{Keywords:}$ Infinite permutation, Permutation complexity, uniformly recurrent, Sturmian words, Thue-Morse word
$\vspace{1.0ex}$

\noindent $\textbf{MSC:}$ 05A05 68R15
\end{abstract}
\section{Introduction}

%
Permutation complexity of aperiodic words is a relatively new notion of word complexity which is based on the idea of an infinite permutation associated to an aperiodic word.  For an infinite aperiodic word $\w$, no two shifts of $\w$ are identical.  Thus, given a linear order on the symbols used to compose $\w$, no two shifts of $\w$ are equal lexicographically.  The infinite permutation associated with $\w$ is the linear order on $\N$ induced by the lexicographic order of the shifts of $\w$.  The permutation complexity of the word $\w$ will be the number of distinct subpermutations of a given length of the infinite permutation associated with $\w$.

Infinite permutations associated with infinite aperiodic words over a binary alphabet act fairly well-behaved, but many of the arguments used for binary words break down when used with words over more than two symbols.  Given a subpermutation of length $n$ of an infinite permutation associated with a binary word, a portion of length $n-1$ of the word can be recovered from the subpermutation.  This is not always the case for subpermutations associated with words over 3 or more symbols.  
For binary words the subpermutations depend on the order on the symbols used to compose $\w$, but the permutation complexity does not depend on the order.  For words over 3 or more symbols, not only do the subpermutations depend on the order on the alphabet but so does the permutation complexity.  
%
%


%

%
%

We start with some basic notation and definitions.  Some properties of infinite permutations are given in Section $\ref{GeneralPermResults}$.  In Section $\ref{SecUnifRecWords}$ we introduce a mapping, $\delta$, on the set of subpermutations of an uniformly recurrent word, and an upper bound for the complexity function is calculated for the image of an aperiodic uniformly recurrent word under the letter doubling map.  We then show that when the mapping $\delta$ is injective it implies that restricting an image of $\delta$ is also injective in Section $\ref{SecRestrictions}$.  The complexity function is established for the image of a Sturmian word in Section $\ref{SecSturmianWords}$, and for the image of the Thue-Morse word in Section $\ref{SecThueMorseWord}$.


\subsection{Words}
A $\textit{word}$ is a finite, (right) infinite, or bi-infinite sequence of symbols taken from a finite non-empty set, $\scr{A}$, called an $\textit{alphabet}$.  The standard operation on words is concatenation, and is represented by juxtaposition of letters and words.  A $\textit{finite word}$ over $\scr{A}$ is a word of the form $u = a_1 a_2 \ldots a_n$ with $n \geq 0$ (if $n=0$ we say $u$ is the $\textit{empty word}$, denoted $\epsilon$) and each $a_i \in \scr{A}$; the $\textit{length}$ of the word $u$ is the number of symbols in the sequence and is denoted by $\abs{u} = n$.  For $a \in \scr{A}$, let $\abs{u}_a$ denote the number of occurrences of the letter $a$ in the word $u$.  The set of all finite words over the alphabet $\scr{A}$ is denoted by $\scr{A}^*$, and is a free monoid with concatenation of words as the operation.  

%
A $\textit{(right)}$ $\textit{infinite}$ $\textit{word}$ over $\scr{A}$ is a word of the form $\w = \w_0 \w_1 \w_2 \ldots$ with each $\w_i \in \scr{A}$, and the set of all infinite words over $\scr{A}$ is denoted $\scr{A}^\N$.  Given $\w \in \scr{A}^* \cup \scr{A}^\N$, any word of the form $u=\w_i\w_{i+1} \ldots \w_{i+n-1}$, with $i \geq 0$, is called a $\textit{factor}$ of $\w$ of length $n \geq 1$.  The set of all factors of a word $\w$ is denoted by $\scr{F}(\w)$.  The set of all factors of length $n$ of $\w$  is denoted $\scr{F}_\w(n)$, and let $\rho_\w(n) = \abs{\scr{F}_\w (n)}$.  The function $\rho_\w: \N \rightarrow \N $ is called the $\textit{factor}$ $\textit{complexity}$ $\textit{function}$, or $\textit{subword}$ $\textit{complexity}$ $\textit{function}$, of $\w$ and it counts the number of factors of length $n$ of $\w$.  For a natural number $i$ we denote by $\w[i] = \w_i\w_{i+1}\w_{i+2}\w_{i+3}\ldots$ the $i \textit{-letter}$ $\textit{shift}$ $\textit{of}$ $\w$.  For natural numbers $i \leq j$, $\w[i,j] = \w_i\w_{i+1}\w_{i+2} \ldots \w_j$ denotes the factor of length $j-i+1$ starting at position $i$ in $\w$.

For words $u \in A^*$ and $v \in A^* \cup A^\N$ where $\w = uv$, we call $u$ a $\textit{prefix}$ of $\w$ and $v$ a $\textit{suffix}$ of $\w$.  A word $\w$ is said to be $\textit{periodic}$ of period $p$ if for each $i \in \N$, $\w_i = \w_{i+p}$, and $\w$ is said to be $\textit{eventually}$ $\textit{periodic}$ of period $p$ if there exists an $N \in \N$ so that for each $i > N$, $\w_i = \w_{i+p}$; or equivalently, $\w$ has a periodic suffix.  A word $\w$ is said to be $\textit{aperiodic}$ if it is not periodic or eventually periodic.

The infinite word $\w \in \scr{A}^\N$ is said to be $\textit{recurrent}$ if for any prefix $p$ of $\w$ there exists a prefix $q$ of $\w$ so that $q = pvp$ for some $v \in \scr{A}^*$.  Equivalently, a word $\w$ is recurrent if each factor of $\w$ occurs infinitely often in $\w$.  The word $\w \in \scr{A}^\N$ is $\textit{uniformly}$ $\textit{recurrent}$ if each factor occurs infinitely often  with bounded gaps.  Thus if $\w$ is uniformly recurrent, for each integer $n > 0$ there is a positive integer $N$ so that for each factor $v$ of $\w$ with $\abs{v} = N$, $\scr{F}_\w(n) \subset \scr{F}(v)$.

%
Let $\scr{A}$ and $\scr{B}$ be two finite alphabets.  A map $\varphi: \scr{A}^* \rightarrow \scr{B}^*$ so that $\varphi(uv) = \varphi(u)\varphi(v)$ for any $u,v \in \scr{A}^*$ is called a $\textit{morphism}$ of $\scr{A}^*$ into $\scr{B}^*$, and $\varphi$ is defined by the image of each letter in $\scr{A}$.  A morphism on $\scr{A}$ is a morphism from $\scr{A}^*$ into $\scr{A}^*$, also called an $\textit{endomorphism}$ of $\scr{A}$.  A morphism $\varphi$ is said to be $\textit{non-erasing}$ if the image of any non-empty word is not empty.  The morphism $d: \scr{A}^* \mapsto \scr{A}^*$ defined by $d(a) = aa$ for each $a \in \scr{A}$ is called the $\textit{letter}$ $\textit{doubling}$ $\textit{map}$. 


\subsection{Permutations from words}
The idea of an infinite permutation that will be here used was introduced in $\cite{FlaFrid}$.  This paper will be dealing with permutation complexity of infinite words so the set used in the following definition will be $\N$ rather than an arbitrary countable set.  To define an $\textit{infinite permutation}$ $\pi$, start with a total order $\prec_\pi$ on $\N$, together with the usual order $<$ on $\N$.  To be more specific, an infinite permutation is the ordered triple $\pi = \left\langle \N,\prec_\pi,<  \right\rangle$, where $\prec_\pi$ and $<$ are total orders on $\N$.  The notation to be used here will be $\pi(i) < \pi(j)$ rather than $i \prec_\pi j.$

Given an infinite aperiodic word $\w = \w_0\w_1\w_2 \ldots$ on an alphabet $\scr{A}$, fix a linear order on $\scr{A}$.  We will use the binary alphabet $\scr{A} = \{0, 1\}$ and use the natural ordering $0<1$.  Once a linear order is set on the alphabet, we can then define an order on the natural numbers based on the lexicographic order of shifts of $\w$.  Considering two shifts of $\w$ with $a \neq b$, $\w[a] = \w_a\w_{a+1}\w_{a+2} \ldots$ and $\w[b] = \w_b\w_{b+1}\w_{b+2} \ldots$, we know that $\w[a] \neq \w[b]$ since $\w$ is aperiodic.  Thus there exists some minimal number $c \geq 0$ so that $\w_{a+c} \neq \w_{b+c}$ and for each $0 \leq i < c$ we have $\w_{a+i} = \w_{b+i}$.  We call $\pi_\w$ the infinite permutation associated with $\w$ and say that $\pi_\w(a) < \pi_\w(b)$ if $\w_{a+c} < \w_{b+c}$, else we say that $\pi_\w(b) < \pi_\w(a)$.  

For natural numbers $a \leq b$ consider the factor $\w[a, b] = \w_a\w_{a+1} \ldots \w_b$ of $\w$ of length $b - a + 1$.  Denote the finite permutation of $\{ 1, 2, \ldots , b - a + 1 \}$ corresponding to the linear order by $\pi_\w[a,b]$.  That is $\pi_\w[a,b]$ is the permutation of $\{ 1, 2, \ldots , b - a + 1 \}$ so that for each $0 \leq i,j \leq (b - a)$, $ \pi_\w[a,b](i) < \pi_\w[a,b](j)$ if and only if $\pi_\w(a + i) < \pi_\w(a + j)$.  Say that $p = p_0p_1 \cdots p_n$ is a $\textit{(finite) subpermutation}$ of $\pi_\w$ if $p = \pi_\w[a,a+n]$ for some $a,n \geq 0$.  For the subpermutation $p = \pi_\w[a,a+n]$ of $\{1, 2, \cdots, n+1 \}$, we say the $\textit{length}$ of $p$ is $n+1$.

Denote the set of all subpermutations of $\pi_\w$ by $\mathrm{Perm}^\w$, and for each positive integer $n$ let 
$$\mathrm{Perm}^\w(n) = \{ \hspace{1.0ex} \pi_\w[i,i+n-1] \hspace{1.0ex} \left| \hspace{1.0ex} i \geq 0 \right. \hspace{1.0ex} \}$$
denote the set of distinct finite subpermutations of $\pi_\w$ of length $n$.  The $\textit{permutation complexity function}$ of $\w$ is defined as the total number of distinct subpermutations of $\pi_\w$ of a length $n$, denoted $\tau_\w(n) = \abs{\mathrm{Perm}^\w(n)}$.

\begin{example}  
Let's consider the well-known Fibonacci word, 
$$t = 0100101001001010010100100101\ldots,$$
with the alphabet $\scr{A} = \{0,1 \}$ ordered as $0 < 1$.  We can see that $t[2] = 001010\ldots $ is lexicographically less than $t[1] = 100101\ldots$, and thus $\pi_t(2) < \pi_t(1)$.

Then for a subpermutation, consider the factor $t[3,5] = 010$.  We see that $\pi_t[3,5] = (231)$ because in lexicographic order we have $\pi_t(5) < \pi_t(3) < \pi_t(4)$.
\end{example}

\section{Some General Permutation Properties}
\label{GeneralPermResults}

Initially work has been done with infinite binary words (see $\cite{AvgFriKamSal, FlaFrid, Makar06, Makar09, Makar10}$).  Suppose $\w = \w_0\w_1\w_2\ldots$ is an aperiodic infinite word over the alphabet $\scr{A}=\{ 0,1 \}$.  First let's look at some remarks about permutations generated by binary words where we use the natural order on $\scr{A}$.
\begin{claim}
\emph{($\cite{Makar06}$)}
\label{PCClaim01}
For an infinite aperiodic word $\w$ over $\scr{A} = \{ 0, 1 \}$ with the natural ordering we have:

(1) $\pi_\w(i) < \pi_\w(i+1)$ if and only if $\w_i = 0$.

(2) $\pi_\w(i) > \pi_\w(i+1)$ if and only if $\w_i = 1$.

(3) If $\w_i = \w_j$, then $\pi_\w(i) < \pi_\w(j)$ if and only if $\pi_\w(i+1) < \pi_\w(j+1)$
\end{claim}
\begin{lemma}
\emph{($\cite{Makar06}$)}
\label{PermComp01}
Given two infinite binary words u = $u_0u_1\ldots$ and $v=v_0v_1 \ldots$ with $\pi_u[0, n+1] = \pi_v[0, n+1]$, it follows that $u[0,n] = v[0,n]$.
\end{lemma}
We do have a trivial upper bound for $\tau_\w(n)$ being the number of permutations of length $n$, which is $n!$.  Lemma $\ref{PermComp01}$ directly implies a lower bound for the permutation complexity for a binary aperiodic word $\w$, namely the factor complexity of $\w$.  Thus, initial bounds on the permutation complexity can be seen to be:
$$ \rho_\w(n-1) \leq \tau_\w(n) \leq  n!$$

For $a \in \scr{A} = \{0,1\}$, let $\bar{a}$ denote the $\textit{complement}$ of $a$, that is $\bar{0} = 1$ and $\bar{1} = 0$.  If $u = u_1u_2u_3 \cdots$ is a word over $\scr{A}$, the $\textit{complement}$ of $u$ is defined to be the word composed of the complement of the letters in $u$, that is $\bar{u} = \bar{u}_1\bar{u}_2\bar{u}_3 \cdots$.  The following lemma shows the relationship of the complexity function between an aperiodic binary word $\w$ and its complement $\overline{\w}$.  This lemma will be used when calculating the permutation complexity of the image of Sturmian words under the doubling map in Section $\ref{SecSturmianWords}$.

\begin{lemma}
\label{PermCompOfCompliment}
Let $\w = \w_0\w_1\w_2\cdots$ be an infinite aperiodic binary word, and let $\overline{\w} = \overline{\w_0}\overline{\w_1}\overline{\w_2}\cdots$ be the complement of $\w$.  For each $n \geq 1$, 
$$\tau_\w(n) = \tau_{\overline{\w}}(n).$$
\end{lemma}
\begin{proof}
For some $a \neq b$, suppose $\w[a] < \w[b]$.  Thus there is some (possibly empty) factor $u$ of $\w$ so that $\w[a] = u0 \cdots $ and $\w[b] = u1 \cdots$.  Then we see $\overline{\w}[a] = \overline{u}1\cdots$ and $\overline{\w}[b] = \overline{u}0\cdots$, so we have $\overline{\w}[a] > \overline{\w}[b]$.

For both $\w$ and $\overline{\w}$ it should be clear $\tau_\w(1) = \tau_{\overline{\w}}(1) = 1$, namely the subpermutation $(1)$.  Let $n \geq 2$.  For a permutation $p$ of $\{1, 2, \ldots, n \}$, define the permutation $\tilde{p}$ of $\{1, 2, \ldots, n \}$ by 
$$\tilde{p}_i = n - p_i +1$$
for each $i$.

Let $p = \pi_\w[a,a+n-1]$ be a subpermutation of $\pi_\w$ and $q = \pi_{\overline{\w}}[a,a+n-1]$ be a subpermutation of $\pi_{\overline{\w}}$.  For each $0 \leq i,j \leq n-1$, $i \neq j$, if $p_i < p_j$ then $q_i > q_j$.

Let $0 \leq i \leq n-1$.  There are $p_i - 1$ many $j$ so that $p_j < p_i$ and there are $n - p_i$ many $j$ so that $p_j > p_i$.  Therefore there are exactly $n - p_i$ many $j$ so that $q_j < q_i$, so $q_i = n - p_i + 1$.  Thus $q = \tilde{p}$ and for any $p \in \mathrm{Perm}^\w(n)$ we have $\tilde{p} \in \mathrm{Perm}^{\overline{\w}}(n)$, so
$$\abs{\mathrm{Perm}^\w(n)} \leq \abs{\mathrm{Perm}^{\overline{\w}}(n)}.$$
By a similar argument we can see $p = \tilde{q}$ and for $q \in \mathrm{Perm}^{\overline{\w}}(n)$ we have $\tilde{q} \in \mathrm{Perm}^\w(n)$, so 
$$\abs{\mathrm{Perm}^{\overline{\w}}(n)} \leq \abs{\mathrm{Perm}^\w(n)}.$$

Therefore $\abs{\mathrm{Perm}^\w(n)} = \abs{\mathrm{Perm}^{\overline{\w}}(n)}$ and $\tau_\w(n) = \tau_{\overline{\w}}(n)$.
$\qed$
\end{proof}

We would like to define some terms that will be used repeatedly in this paper.

\begin{definition}
\label{SameForm}
Two permutations $p$ and $q$ of $\{1, 2, \ldots, n  \}$ have the $\textit{same form}$ if for each $i = 0, 1, \ldots, n-1$, $p_i < p_{i+1}$ if and only if $q_i < q_{i+1}$.  For a binary word $u$ of length $n-1$, say that $p$ $\textit{has}$ $\textit{form}$ $u$ if 
$$p_i<p_{i+1} \Longleftrightarrow u_i = 0$$
for each $i = 0, 1, \ldots, n-2$.
\end{definition}

%
\begin{definition}
Let $p = \pi[a,a+n]$ be a subpermutation of the infinite permutation $\pi$.  The $\textit{left restriction of}$ $p$, denoted by $L(p)$, is the subpermutation of $p$ so that $L(p) = \pi[a, a+n-1]$.  The $\textit{right restriction of}$ $p$, denoted by $R(p)$, is the subpermutation of $p$ so that $R(p) = \pi[a+1, a+n]$.  The $\textit{middle restriction of}$ $p$, denoted by $M(p)$, is the subpermutation of $p$ so that $M(p) = R(L(p)) = L(R(p)) = \pi[a+1, a+n-1]$.
\end{definition}

For each $i$, there are $p_i-1$ terms in $p$ that are less than $p_i$ and there are $n-p_i$ terms that are greater than $p_i$.  Thus consider some $0 \leq i \leq n-1$ and the values of $L(p)_i$ and $R(p)_i$.  If $p_0 < p_{i+1}$ there will be $p_{i+1}-2$ terms in $R(p)$ less than $R(p)_i$ so we have $R(p)_i = p_{i+1}-1$.  In a similar sense, if $p_n < p_i$ we have $L(p)_i = p_i - 1$.  If $p_0 > p_{i+1}$ there will be $p_{i+1}-1$ terms in $R(p)$ less than $R(p)_i$ so we have $R(p)_i = p_{i+1}$.  In a similar sense, if $p_n > p_i$ we have $L(p)_i = p_i $.  

The values in $M(p)$ can be found by finding the values in $R(L(p))$ or $L(R(p))$.  Since $R(L(p))$ or $L(R(p))$ correspond to the same subpermutation of $p$, $R(L(p))_i < R(L(p))_j$ if and only if $L(R(p))_i < L(R(p))_j$.  Therefore $R(L(p)) = L(R(p))$.

It should also be clear that if there are two subpermutations $p= \pi_T[a,a+n]$ and $q = \pi_T[b,b+n]$ so that $p=q$ then $L(p) = L(q)$, $R(p) = R(q)$, and $M(p) = M(q)$.

\section{Uniformly Recurrent Words}
\label{SecUnifRecWords}

Let $\w$ be an aperiodic infinite uniformly recurrent word over $\scr{A} = \{0,1 \}$, and $\pi_\w$ be the infinite permutation associated with $\w$ using the natural order on the alphabet.  We would like to describe the infinite permutation associated with $d(\w)$, the image of $\w$ under the doubling map.  If $u = \w[a,a+n-1]$ is a factor of $\w$ of length $n$, it is helpful to note $d(u) = d(\w)[2a, 2a+2n-1]$ will be a factor of $d(\w)$ of length $2n$.

Since $\w$ is a uniformly recurrent word it will not contain arbitrarily long strings of contiguous $0$ or $1$.  Thus there are $k_0, k_1 \in \N$ so that 
\begin{align*}
10^{k_0}1 \\
01^{k_1}0
\end{align*}
are factors of $\w$, but $0^{k_0+1}$ and $1^{k_1+1}$ are not.  We then define the following classes of words:
\begin{align*}
C_0 &= 0^{k_0} \\
C_1 &= 0^{k_0-1}1 \\
C_2 &= 0^{k_0-2}1 \\
&\vdots \\
C_{k_0-1} &= 01 \\
C_{k_0} &= 10 \\
C_{k_0+1} &= 1^2 0 \\
&\vdots \\
C_{k_0+k_1-1} &= 1^{k_1}.
\end{align*}
For each $i \in \N$, $\w[i] = \w_i \w_{i+1} \cdots$ can have exactly one the above classes of words as a prefix.  It should be clear $C_0 < C_1 < \cdots < C_{k_0+k_1-1}$, and so $d(C_i) < d(C_j)$ for $i < j$ since the doubling map $d$ is order preserving, as shown in Lemma $\ref{ImageOfTheCiClasses}$.  The next lemma will not only show that the doubling map is an order preserving map, but also the order of the image of $\w_i$ under the doubling map.
\begin{lemma}
\label{ImageOfTheCiClasses}
Let $\w$ be as above.  Suppose $\w[a]$ and $\w[b]$ are two shifts of $\w$ for some $a \neq b$ so that $\w[a] < \w[b]$.  Moreover, suppose $C_i$ is a prefix of $\w[a]$ and $C_j$ is a prefix of $\w[b]$ where $i \leq j$.  Then $d(\w[a]) < d(\w[b])$, and 
\begin{itemize}
\item[(a)] If $\w_a = \w_b = 0$ and $i < j$, then $d(\w)[2a] < d(\w)[2a+1] < d(\w)[2b] < d(\w)[2b+1]$.
\item[(b)] If $\w_a = \w_b = 0$ and $i = j$, then $d(\w)[2a] < d(\w)[2b] < d(\w)[2a+1] < d(\w)[2b+1]$.
\item[(c)] If $\w_a = 0$ and $\w_b = 1$, then $d(\w)[2a] < d(\w)[2a+1] < d(\w)[2b+1] < d(\w)[2b]$.
\item[(d)] If $\w_a = \w_b = 1$and $i < j$, then $d(\w)[2a+1] < d(\w)[2a] < d(\w)[2b+1] < d(\w)[2b]$.
\item[(e)] If $\w_a = \w_b = 1$and $i = j$, then $d(\w)[2a+1] < d(\w)[2b+1] < d(\w)[2a] < d(\w)[2b]$.
\end{itemize}
\end{lemma}
\begin{proof}
Since $\w[a] < \w[b]$, there is some (possibly empty) factor $u$ of $\w$ so that 
$$ \w[a] = u0 \cdots $$
$$ \w[b] = u1 \cdots $$
and thus 
$$ d(\w[a]) = d(u)00 \cdots $$
$$ d(\w[b]) = d(u)11 \cdots $$
so $d(\w[a]) < d(\w[b])$ and $d$ is an order preserving map.

Each of the cases will be looked at independently.

$\vspace{1.0ex}$

$\textbf{(a)}$  Suppose $\w_a = \w_b = 0$ and $i < j$. Since both $\w[a]$ and $\w[b]$ start with 0, $\w[a]$ has $0^{k_0 -i}1$ as a prefix and $\w[b]$ has $0^{k_0 -j}1$ as a prefix.  Thus $d(\w)[2a]$ has $0^{2(k_0 -i)}1$ as a prefix and $d(\w)[2b]$ has $0^{2(k_0 -j)}1$ as a prefix, and 
$$ 0^{2(k_0 -i)}1 < 0^{2(k_0 -i)-1}1 < 0^{2(k_0 -j)}1 < 0^{2(k_0 -j)-1}1. $$

$\vspace{1.0ex}$

$\textbf{(b)}$  Suppose $\w_a = \w_b = 0$ and $i = j$. Since both $\w[a]$ and $\w[b]$ start with 0, for $n = k_0 -i$, $\w[a]$ and $\w[b]$ have $0^{n}1$ as a prefix.  Thus $d(\w)[2a]$ and $d(\w)[2b]$ have $0^{2n}1$ as a prefix.  Since $\w[a] < \w[b]$ is given and $d$ is an order preserving map we have 
\begin{align*}
d(\w)[2a] &< d(\w)[2b] \\
d(\w)[2a+1] &< d(\w)[2b+1] \\
0^{2n}1 &< 0^{2n-1}1.
\end{align*}
Thus $d(\w)[2a] < d(\w)[2b] < d(\w)[2a+1] < d(\w)[2b+1]$.

$\vspace{1.0ex}$

$\textbf{(c)}$  Suppose $\w_a = 0$ and  $\w_b = 1$, so $i < j$. Since $\w[a]$ start with 0 and $\w[b]$ start with 1, there are numbers $n$ and $m$ so that $\w[a]$ has $0^n1$ as a prefix and $\w[b]$ has $1^m0$ as a prefix.  Thus $d(\w)[2a]$ has $0^{2n}1$ as a prefix and $d(\w)[2b]$ has $1^{2m}0$ as a prefix, and 
$$ 0^{2n}1 < 0^{2n-1}1 < 1^{2m-1}0 < 1^{2m}0. $$ 

$\vspace{1.0ex}$

$\textbf{(d)}$  Suppose $\w_a = \w_b = 1$ and $i < j$. Since both $\w[a]$ and $\w[b]$ start with 1, $\w[a]$ has $1^{i-k_0+1}0$ as a prefix and $\w[b]$ has $1^{j-k_0+1}0$ as a prefix.  Thus $d(\w)[2a]$ has $1^{2(i-k_0)+2}0$ as a prefix and $d(\w)[2b]$ has $1^{2(j-k_0)+2}0$ as a prefix, and 
$$ 1^{2(i-k_0)+1}0 < 1^{2(i-k_0)+2}0 < 1^{2(j-k_0)+1}0 < 1^{2(j-k_0)+2}0. $$ 

$\vspace{1.0ex}$

$\textbf{(e)}$  Suppose $\w_a = \w_b = 1$ and $i = j$. Since both $\w[a]$ and $\w[b]$ start with 1, for $n = i-k_0+1$, $\w[a]$ and $\w[b]$ have $1^{n}0$ as a prefix.  Thus $d(\w)[2a]$ and $d(\w)[2b]$ have $1^{2n}0$ as a prefix.  Since $\w[a] < \w[b]$ is given and $d$ is an order preserving map we have 
\begin{align*}
d(\w)[2a] &< d(\w)[2b] \\
d(\w)[2a+1] &< d(\w)[2b+1] \\
1^{2n - 1}0 &< 1^{2n}0.
\end{align*}
Thus $d(\w)[2a+1] < d(\w)[2b+1] < d(\w)[2a] < d(\w)[2b]$.
$\qed$
\end{proof}

For $k = \sup \{ k_0,k_1 \}$, there is an $N_k$ so any factor $u$ of $\w$ of length $n \geq N_k$ will contain all factors of length $k$ as a subword, and so $u$ will have $C_j$ as a subword for each $j$.  One note about the factors of $d(\w)$.  For $n \geq N_k$ and two factors $u = d(\w)[2x,2x+2n]$ and $v = d(\w)[2y+1,2y+2n+1]$ of $d(\w)$, then $u \neq v$.  This is because a prefix of $u$ will begin with an even number of one letter (either $0^{2m}1$ or $1^{2m}0$ for some $m$), and a prefix of $v$ will begin with an odd number of one letter (either $0^{2m+1}1$ or $1^{2m+1}0$ for some $m$).

Let $u$ be a factor of $\w$ of length $n \geq N_k$.  There is an $a$ so that $u = \w[a,a+n-1]$.  For each $0 \leq i \leq n-1$ there is one $j$ so that $\w[a+i]$ has $C_j$ as a prefix.  In the factor $\w[a,a+n+k-2]$ of length $n+k-1$, we will know explicitly which $C_j$ is a prefix of the shift $\w[a+i]$ for each $0 \leq i \leq n-1$.  Let $p = \pi_\w[a,a+n+k-1]$ be a subpermutation of $\pi_\w$ of length $n+k$.  The factor $\w[a,a+n+k-2]$ of length $n+k-1$ is the form of $p$, and has $u$ as a prefix.  

For each $j \in \{0, 1, \ldots, k_0+k_1-1  \}$ define
$$ \gamma_j = \left\{ \hspace{0.5ex} 0 \leq i \leq n-1 \hspace{0.5ex} \left| \hspace{0.5ex} C_j \text{ is a prefix of } \w[a+i] \hspace{0.5ex} \right.\right\}. $$
So $\abs{\gamma_0} + \abs{\gamma_1} + \cdots + \abs{\gamma_{k_0+k_1-1}} = n$ and $\gamma_i \cap \gamma_j = \emptyset$ for $i \neq j$.  Since $\abs{u} \leq N_k$, we know $\abs{\gamma_j} \geq 1$ for each $j$.  We can see $d(u) = d(\w)[2a, 2a+2n-1]$, and let $p'$ be the subpermutation $p'= \pi_{d(\w)}[2a, 2a+2n-1]$.  Using Lemma $\ref{ImageOfTheCiClasses}$ and the size of each of the $\gamma_j$ sets we can determine the values of $p'$ based on the values of $L^k(p)$, the $k$-left restriction of $p$.  For each $j \in \{0, 1, \ldots, k_0+k_1-1  \}$ define
$$ S_j = \sum_{i=0}^j \abs{\gamma_i} $$
and say $S_{-1} = 0$.

\begin{proposition}
\label{CalcTheFwdImage}
Let $\w$, $u$, $p$, and $p'$ be as above.  For each $0 \leq i \leq n-1$, there is a $j$ so $\w[a+i]$ has $C_j$ as a prefix. 
\begin{itemize}
\item[(a)] If $p_i < p_{i+1}$ then $p'_{2i} = L^{k}(p)_i + S_{j-1}$ and $p'_{2i+1} = L^{k}(p)_i + S_{j}$
\item[(b)] If $p_i > p_{i+1}$ then $p'_{2i} = L^{k}(p)_i + S_{j} $ and $p'_{2i+1} = L^{k}(p)_i + S_{j-1}$
\end{itemize}
\end{proposition}
\begin{proof}
Let $0 \leq i \leq n-1$ and suppose that $C_j$ is a prefix of $\w[a+i]$ for some $0 \leq j \leq k_0+k_1-1$.  

$\vspace{1.0ex}$

$\textbf{(a)}$  Suppose $p_i < p_{i+1}$, and so $\w_{a+i} = u_i = 0$.  

For $p'_{2i}$, there are $L^k(p)_i -1$ many $h$ so that $p_i > p_h$, and thus $L^k(p)_i - 1$ many $h$ so that $p'_{2i} > p'_{2h}$.  Likewise there are $n - L^k(p)_i$ many $h$ so that $p'_{2i} < p'_{2h}$.  By Lemma $\ref{ImageOfTheCiClasses}$ if $m < j$ and $h \in \gamma_m$ then $p'_{2i} > p'_{2h+1}$, and if $m \geq j$ and $h \in \gamma_m$ then $p'_{2i} < p'_{2h+1}$.  Thus there are $S_{j-1}$ many $h$ so that $p'_{2i} > p'_{2h+1}$, and likewise there are $n-S_{j-1}$ many $h$ so that $p'_{2i} < p'_{2h+1}$.  Therefore there are exactly $L^k(p)_i - 1 + S_{j-1}$ many $h$ so that $p'_{2i} > p'_h$, and
$$p'_{2i} = L^k(p)_i - 1 + S_{j-1} + 1 = L^k(p)_i + S_{j-1}.$$

For $p'_{2i+1}$, since there are $L^k(p)_i -1$ many $h$ so that $p_i > p_h$, there are $L^k(p)_i - 1$ many $h$ so that $p'_{2i} > p'_{2h}$ and $p'_{2i+1} > p'_{2h+1}$.  Likewise there are $n - L^k(p)_i$ many $h$ so that $p'_{2i+1} < p'_{2h+1}$.  By Lemma $\ref{ImageOfTheCiClasses}$ if $m \leq j$ and $h \in \gamma_m$ then $p'_{2i+1} > p'_{2h}$, and if $m > j$ and $h \in \gamma_m$ then $p'_{2i+1} < p'_{2h}$.  Thus there are $S_{j}$ many $h$ so that $p'_{2i+1} > p'_{2h}$, and likewise there are $n-S_{j}$ many $h$ so that $p'_{2i+1} < p'_{2h}$.  Therefore there are exactly $L^k(p)_i - 1 + S_{j}$ many $h$ so that $p'_{2i} > p'_h$, so
$$p'_{2i+1} = L^k(p)_i - 1 + S_{j} + 1 = L^k(p)_i + S_{j}.$$

$\vspace{1.0ex}$

$\textbf{(b)}$  Suppose $p_i > p_{i+1}$, and so $\w_{a+i} = u_i = 1$.

For $p'_{2i}$, there are $L^k(p)_i -1$ many $h$ so that $p_i > p_h$, and thus $L^k(p)_i - 1$ many $h$ so that $p'_{2i} > p'_{2h}$.  Likewise there are $n - L^k(p)_i$ many $h$ so that $p'_{2i} < p'_{2h}$.  By Lemma $\ref{ImageOfTheCiClasses}$ if $m \leq j$ and $h \in \gamma_m$ then $p'_{2i} > p'_{2h+1}$, and if $m > j$ and $h \in \gamma_m$ then $p'_{2i} < p'_{2h+1}$.  Thus there are $S_{j}$ many $h$ so that $p'_{2i} > p'_{2h+1}$, and likewise there are $n-S_{j}$ many $h$ so that $p'_{2i} < p'_{2h+1}$.  Therefore there are exactly $L^k(p)_i - 1 + S_{j}$ many $h$ so that $p'_{2i} > p'_h$, and
$$p'_{2i} = L^k(p)_i - 1 + S_{j} + 1 = L^k(p)_i + S_{j}.$$

For $p'_{2i+1}$, since there are $L^k(p)_i -1$ many $h$ so that $p_i > p_h$, there are $L^k(p)_i - 1$ many $h$ so that $p'_{2i} > p'_{2h}$ and $p'_{2i+1} > p'_{2h+1}$.  Likewise there are $n - L^k(p)_i$ many $h$ so that $p'_{2i+1} < p'_{2h+1}$.  By Lemma $\ref{ImageOfTheCiClasses}$ if $m < j$ and $h \in \gamma_m$ then $p'_{2i+1} > p'_{2h}$, and if $m \geq j$ and $h \in \gamma_m$ then $p'_{2i+1} < p'_{2h}$.  Thus there are $S_{j-1}$ many $h$ so that $p'_{2i+1} > p'_{2h}$, and likewise there are $n-S_{j-1}$ many $h$ so that $p'_{2i+1} < p'_{2h}$.  Therefore there are exactly $L^k(p)_i - 1 + S_{j-1}$ many $h$ so that $p'_{2i+1} > p'_h$, and
$$p'_{2i+1} = L^k(p)_i - 1 + S_{j-1} + 1 = L^k(p)_i + S_{j-1}.$$
$\qed$
\end{proof}

The following corollaries show some nice properties that follow from Proposition $\ref{CalcTheFwdImage}$.  The first corollary ($\ref{Cor_pNEQq_SameFormAndSameRest}$) gives an example of when distinct subpermutations of $\pi_\w$ will lead to the same subpermutation of $\pi_{d(\w)}$.  The next corollary ($\ref{Corr_uNEQv}$) shows when two subpermutations of $\pi_\w$ will definitely lead to distinct subpermutations of $\pi_{d(\w)}$.
\begin{corollary}
\label{Cor_pNEQq_SameFormAndSameRest}
Let $\w$ be as defined above.  If $\pi_\w[a,a+n+k-1]$ and $\pi_\w[b,b+n+k-1]$, $a \neq b$, are subpermutations of $\pi_\w$ where $\pi_\w[a,a+n-1] = \pi_\w[b,b+n-1]$ and for each $0 \leq i \leq n-1$, there is some $j$ so that both $\w[a+i]$ and $\w[b+i]$ have $C_j$ as a prefix.  Then $\pi_{d(\w)}[2a,2a+2n-1] = \pi_{d(\w)}[2b,2b+2n-1]$.
\end{corollary}
\begin{proof}
Let $p = \pi_\w[a,a+n+k-1]$ and $q = \pi_\w[b,b+n+k-1]$, $a \neq b$, with $p$ and $q$ as in the hypothesis.  For each $0 \leq i \leq n-1$, $L^k(p)_i = L^k(q)_i$ and each of $\w[a+i]$ and $\w[b+i]$ have the same $C_j$ as a prefix for some $j$, so $p'_{2i} = q'_{2i}$ and $p'_{2i+1} = q'_{2i+1}$ by Proposition $\ref{CalcTheFwdImage}$ so $p' = q'$.
$\qed$
\end{proof}

\begin{corollary}
\label{Corr_uNEQv}
Let $\w$ be as defined above.  If $p = \pi_\w[a,a+n+k-1]$ and $q = \pi_\w[a,a+n+k-1]$ are subpermutations of $\pi_\w$ where one of the following conditions is true
\begin{itemize}
\item[(a)] $\w[a,a+n-1] \neq \w[b,b+n - 1]$
\item[(b)] $L^k(p) \neq L^k(q)$
\end{itemize}
then $p' \neq q'$.
\end{corollary}
\begin{proof}

$\textbf{(a)}$  Since $\w[a,a+n-1] \neq \w[b,b+n-1]$, then there is an $0 \leq i \leq n-1$ so that, without loss of generality, $\w_{a+i} = 0$ and $\w_{b+i} = 1$.  Thus $d(\w)[2a+2i, 2a+2i+1] = 00$ and $d(\w)[2b+2i, 2b+2i+1] = 11$ so $p'_{2i} < p'_{2i+1}$ and $q'_{2i} > q'_{2i+1}$, and $p' \neq q'$.


$\textbf{(b)}$  Since $L^k(p) \neq L^k(q)$, then there are $0 \leq i,j \leq n-1$, $i \neq j$, so that, without loss of generality, $L^k(p)_i < L^k(p)_j$ and $L^k(q)_i > L^k(q)_j$, so $\w[a+i] < \w[a+j]$ and $\w[b+i] > \w[b+j]$.  Thus $d(\w)[2a+2i] < d(\w)[2a+2j]$ and $d(\w)[2b+2i] > d(\w)[2b+2j]$ so $p'_{2i} < p'_{2j}$ and $q'_{2i} > q'_{2j}$, and $p' \neq q'$.
$\qed$
\end{proof}

Fix a subpermutation $p = \pi_\w[a,a+n+k-1]$, and let $p' = \pi_{d(\w)}[2a,2a+2n-1]$.  The terms of $p'$ can be defined using the method given in Proposition $\ref{CalcTheFwdImage}$.  Let $q=\pi_\w[b,b+n+k-1]$, $b \neq a$, be a subpermutation of $\pi_\w$ and let $q'=\pi_{d(\w)}[2b,2b+2n-1]$ as in Proposition $\ref{CalcTheFwdImage}$.  The following lemma shows that if $p=q$ we know $p' = q'$, but the converse of this is not necessarily true.  The objective here is using the idea of $p'$ to define a map from the set of subpermutations of $\pi_\w$ to the set of subpermutations of $\pi_{d(\w)}$, and this map will be well-defined by Proposition $\ref{CalcTheFwdImage}$.

\begin{lemma}
\label{pISqTHENppISqp}
If $p = q$, then $p' = q'$.
\end{lemma}
\begin{proof}
Suppose that $p=q$.  So $p_i = q_i$ and thus $L^k(p)_i = L^k(q)_i$ for each $0 \leq i \leq n-1$.  Since $p=q$, $p$ and $q$ have the same form, so $\w[a,a+n+k-1] = \w[b,b+n+k-1]$ and if $\w[a+i]$ has $C_j$ as a prefix, for some $j$, then $\w[b+i]$ has $C_j$ as a prefix as well.  Thus by Corollary $\ref{Cor_pNEQq_SameFormAndSameRest}$, $p' = q'$.
$\qed$
\end{proof}

%
%
Thus there is a well-defined function from the set of subpermutations of $\pi_\w$ to the set of subpermutations of $\pi_{d(\w)}$.  Let $p = \pi_\w[a,a+n+k-1]$, and define $\delta(p) = p' = \pi_{d(\w)}[2a,2a+2n-1]$ using the formula in Proposition $\ref{CalcTheFwdImage}$.  Thus we have the map 
$$\delta : \mathrm{Perm}^\w(n+k) \rightarrow \mathrm{Perm}^{d(\w)}(2n)$$
Not all subpermutations of $\pi_\w$ will be the image under $\delta$ of another subpermutation.

Let $n > 2N_k$ and $a$ be natural numbers.  Then $n$ and $a$ can be either even or odd, and for the subpermutation $\pi_{d(\w)}[a,a+n-1]$, there exist natural numbers $b$ and $m$ so that one of 4 cases hold:
\begin{enumerate}
\item $\pi_{d(\w)}[a,a+n] = \pi_{d(\w)}[2b,2b+2m]$, even starting position with odd length.
\item $\pi_{d(\w)}[a,a+n] = \pi_{d(\w)}[2b,2b+2m-1]$, even starting position with even length.
\item $\pi_{d(\w)}[a,a+n] = \pi_{d(\w)}[2b+1,2b+2m]$, odd starting position with even length.
\item $\pi_{d(\w)}[a,a+n] = \pi_{d(\w)}[2b+1,2b+2m-1]$, odd starting position with odd length.
\end{enumerate}

Consider two subpermutations $\pi_{d(\w)}[2c, 2c+n]$ and $\pi_{d(\w)}[2d+1, 2d+n+1]$, with $n > 2N_k$.  The subpermutation $\pi_{d(\w)}[2c, 2c+n]$ will have form $d(\w)[2c, 2c+n-1]$, and $\pi_{d(\w)}[2d+1, 2d+n+1]$ will have form $d(\w)[2d+1, 2d+n]$.  Since the length of these factors is at least $2N_k$, we know $d(\w)[2c, 2c+n-1] \neq d(\w)[2d+1, 2d+n]$, and thus $\pi_{d(\w)}[2c, 2c+n] \neq \pi_{d(\w)}[2d+1, 2d+n+1]$ because they do not have the same form.  Thus we can break up the set $\mathrm{Perm}^{d(\w)}(n)$ into two classes of subpermutations, namely the subpermutations that start at an even position or an odd position.  So say that $\mathrm{Perm}^{d(\w)}_{ev}(n)$ is the set of subpermutations $p$ of length $n$ so that $p = \pi_{d(\w)}[2b,2b+n-1]$ for some $b$, and that $\mathrm{Perm}^{d(\w)}_{odd}(n)$ is the set of subpermutations $p$ of length $n$ so that $p = \pi_{d(\w)}[2b+1,2b+n]$ for some $b$.  Thus
$$\mathrm{Perm}^{d(\w)}(n) = \mathrm{Perm}^{d(\w)}_{ev}(n) \cup \mathrm{Perm}^{d(\w)}_{odd}(n),$$
where   
$$\mathrm{Perm}^{d(\w)}_{ev}(n) \cap \mathrm{Perm}^{d(\w)}_{odd}(n) = \emptyset.$$

Thus for $n \geq N_k$, $\mathrm{Perm}^{d(\w)}_{ev}(2n)$ is the set of all subpermutations of length $2n$ starting at an even position.  So for $\pi_{d(\w)}[2a,2a+2n-1]$, we have the subpermutation $p = \pi_\w[a,a+n+k-1]$, and $\delta(p) = p' = \pi_{d(\w)}[2a,2a+2n-1]$.  Thus the map 
$$\delta : \mathrm{Perm}^\w(n+k) \mapsto \mathrm{Perm}^{d(\w)}_{ev}(2n)$$
is a surjective map.  

For $p=\pi_{d(\w)}[a,a+n+k-1]$, we can then define three additional maps by looking at the left, right, and middle restrictions of $\delta(p) = p'$.  These maps are
\begin{align*}
\delta_L  :  \mathrm{Perm}^\w(n+k) &\mapsto \mathrm{Perm}^{d(\w)}_{ev}(2n-1) \\
\delta_R : \mathrm{Perm}^\w(n+k) &\mapsto \mathrm{Perm}^{d(\w)}_{odd}(2n-1) \\
\delta_M : \mathrm{Perm}^\w(n+k) &\mapsto \mathrm{Perm}^{d(\w)}_{odd}(2n-2) 
\end{align*}
and are defined by
\begin{align*}
\delta_L(p) &= L(\delta(p)) = L(p')\\
\delta_R(p) &= R(\delta(p)) = R(p')\\
\delta_M(p) &= M(\delta(p)) = M(p')
\end{align*}
It can be readily verified that these three maps are surjective.  To see an example of this, consider the map $\delta_L$, and let $\pi_{d(\w)}[2b,2b+2n-2]$ be a subpermutation of $\pi_{d(\w)}$ in $\mathrm{Perm}^{d(\w)}_{ev}(2n-1)$.  Then for the subpermutation $p=\pi_\w[b,b+n+k-1]$, $\delta_L(p) = L(p') = \pi_{d(\w)}[2b,2b+2n-2]$ so $\delta_L$ is surjective.  A similar argument will show that $\delta_R$ and $\delta_M$ are also surjective.  

\begin{lemma}
\label{UpperBoundForTau}
For $n \geq N_k$:
\begin{align*}
\tau_{d(\w)}(2n-1) &\leq  2(\tau_\w(n+k)) \\
\tau_{d(\w)}(2n) &\leq \tau_\w(n+k) + \tau_\w(n+k+1) 
\end{align*}
\end{lemma}
\begin{proof}
Let $n \geq N_k$.  We have:
\begin{align*}
\abs{\mathrm{Perm}^{d(\w)}_{ev}(2n-1)} &\leq \abs{\mathrm{Perm}^\w(n+k)} \\
\abs{\mathrm{Perm}^{d(\w)}_{odd}(2n-1)} &\leq \abs{\mathrm{Perm}^\w(n+k)} \\
\\
\abs{\mathrm{Perm}^{d(\w)}_{ev}(2n)} &\leq \abs{\mathrm{Perm}^\w(n+k)} \\
\abs{\mathrm{Perm}^{d(\w)}_{odd}(2n)} &\leq \abs{\mathrm{Perm}^\w(n+k+1)} 
\end{align*}
since the maps $\delta$, $\delta_L$, $\delta_R$, and $\delta_M$ are all surjective.  Thus we have the following inequalities:
\begin{align*}
\tau_{d(\w)}(2n-1) &= \abs{\mathrm{Perm}^{d(\w)}(2n-1)} = \abs{\mathrm{Perm}^{d(\w)}_{ev}(2n-1)} + \abs{\mathrm{Perm}^{d(\w)}_{odd}(2n-1)} \\
 &\leq \abs{\mathrm{Perm}^\w(n+k)} + \abs{\mathrm{Perm}^\w(n+k)} = 2(\tau_\w(n+k)) \\
\\
\tau_{d(\w)}(2n) &= \abs{\mathrm{Perm}^{d(\w)}(2n)} = \abs{\mathrm{Perm}^{d(\w)}_{ev}(2n)} + \abs{\mathrm{Perm}^{d(\w)}_{odd}(2n)} \\
 &\leq \abs{\mathrm{Perm}^\w(n+k)} + \abs{\mathrm{Perm}^\w(n+k+1)} = \tau_\w(n+k) + \tau_\w(n+k+1) 
\end{align*}
$\qed$
\end{proof}

The maps $\delta$, $\delta_L$, $\delta_R$, and $\delta_M$ can be, but are not necessarily, injective maps.  To see this, consider the following example.  For this example we will use the Thue-Morse word $T$, defined in Section $\ref{SecThueMorseWord}$, and subpermutations of $\pi_T$, the infinite permutation associated with $T$.  There will be 4 classes of $C_j$ words for the Thue-Morse word (namely $C_0 = 00$, $C_1 = 01$, $C_2 = 10$, and $C_3 = 11$), and any factor of length $n \geq 9$ will contain each of these 4 classes.  The following example will use subpermutations of length 9, with $n=7$ and $k=2$, to keep the example subpermutations short.  Examples like this (as in Corollary $\ref{Cor_pNEQq_SameFormAndSameRest}$) can be found for subpermutations of $\pi_T$ of length $2^r+1$ for any $r \geq 3$.

Let $p=\pi_T[0,8] = (4 \hspace{.5ex} 9 \hspace{.5ex} 7 \hspace{.5ex} 2 \hspace{.5ex} 6 \hspace{.5ex} 1 \hspace{.5ex} 3 \hspace{.5ex} 8 \hspace{.5ex} 5)$ and $q=\pi_T[12,20] = (5 \hspace{.5ex} 9 \hspace{.5ex} 7 \hspace{.5ex} 2 \hspace{.5ex} 6 \hspace{.5ex} 1 \hspace{.5ex} 3 \hspace{.5ex} 8 \hspace{.5ex} 4)$.  So $p \neq q$ and both of these subpermutations have form $T[0,7] = T[12,19] = 01101001$.  Then applying the map $\delta$ we see:
$$ p' = \delta(p) = (5 \hspace{.5ex} 8 \hspace{.5ex} 14 \hspace{.5ex} 13 \hspace{.5ex} 12 \hspace{.5ex} 10 \hspace{.5ex} 3 \hspace{.5ex} 6 \hspace{.5ex} 11 \hspace{.5ex} 9 \hspace{.5ex} 1 \hspace{.5ex} 2 \hspace{.5ex} 4 \hspace{.5ex} 7) = \delta(q) = q' $$
So $p' = q'$ which implies $\delta_L(p) = \delta_L(q)$, $\delta_R(p) = \delta_R(q)$, and $\delta_M(p) = \delta_M(q)$.  Thus these 4 maps are not injective in general and the values in Lemma $\ref{UpperBoundForTau}$ are only an upper bound.

\section{Injective Restriction Mappings}
\label{SecRestrictions}

In this section we will investigate when the restriction mappings are injective.  If $\delta$ is not injective, then $\delta_R$, $\delta_L$, and $\delta_M$ will not be injective.  But when $\delta$ is injective it implies $\delta_R$ and $\delta_L$ are injective in general, as shown by Proposition $\ref{DubRestAreBijection}$.  Unfortunately, this does not imply that the map $\delta_M$ is injective, as can be seen in Lemma $\ref{ThueMorseDubRestAreBijection}$.

%
%
%
%
%
%

\begin{lemma}
\label{DiffCjSuffixSameLk}
For the word $\w$, let $p = \pi_\w[a,a+n+k-1]$, $q = \pi_\w[b,b+n+k-1]$, $p'$, and $q'$ be as above.  Suppose $L^k(p) = L^k(q)$, but $\w[a+n-1]$ and $\w[b+n-1]$ each have a different $C_j$ class as a prefix and $\w[a+n-1] < \w[b+n-1]$.  Then there is a $j$ so that $\w[a+n-1]$ has $C_j$ as a prefix and $\w[b+n-1]$ has $C_{j+1}$ as a prefix.  Moreover, $\abs{p'_{2n-2} - q'_{2n-2}} \geq 1$ and $\abs{p'_{2n-1} - q'_{2n-1}} \geq 1$.
\end{lemma}
\begin{proof}
Let $u = \w[a,a+n-1]$ and $v = \w[b,b+n-1]$.  Since $L^k(p) = L^k(q)$ we know for each $0 \leq i \leq n-2$ $L^k(p)_i < L^k(p)_{i+1}$ if and only if $L^k(q)_i < L^k(q)_{i+1}$, so $u_i = v_i$ and thus $u[0, n-2] = v[0, n-2]$.  

We will use the following notation
$$ U_j = \{ \hspace{0.5ex} 0 \leq i \leq n-1 \hspace{0.5ex} | \hspace{0.5ex} \w[a+i] \text{ has } C_j \text{ as a prefix.} \hspace{0.5ex} \} $$
$$ V_j = \{ \hspace{0.5ex} 0 \leq i \leq n-1 \hspace{0.5ex} | \hspace{0.5ex} \w[b+i] \text{ has } C_j \text{ as a prefix.} \hspace{0.5ex} \} $$
and due to the length of $u$ and $v$ we know $\abs{U_j} \geq 1$ and $\abs{V_j} \geq 1$ for each $j$.

%
First we will show there is a $j$ so that $\w[a+n-1]$ has $C_j$ as a prefix and $\w[b+n-1]$ has $C_{j+1}$ as a prefix.  Assume there is an $i > 1$ so that $\w[a+n-1]$ has $C_j$ as a prefix and $\w[b+n-1]$ has $C_{j+i}$ as a prefix.  Then there is an $l$ so that $\w[a + l]$ and $\w[b+l]$ each have $C_{j+1}$ as a prefix.  Thus we see that $L^k(p)_l > L^k(p)_{n-1}$ and $L^k(q)_l < L^k(q)_{n-1}$, which implies $L^k(p) \neq L^k(q)$.  Thus we find a contradiction to the assumption, so there is some $j$ so that $\w[a+n-1]$ has $C_j$ as a prefix and $\w[b+n-1]$ has $C_{j+1}$ as a prefix.

%
$\vspace{0.5ex}$

We now show that the difference in the values in $p'$ and $q'$ differs by at least 1.  Since $\w[a+n-1]$ has $C_j$ as a prefix for each $l \in U_{j+1}$ we know $p_{n-1} < p_l$, and so $q_{n-1} < q_l$ since $L^k(p) = L^k(q)$.  Likewise, since $\w[b+n-1]$ has $C_{j+1}$ as a prefix, for each $l \in V_{j}$ we know $q_{n-1} > q_l$ and so $p_{n-1} > p_l$.  Thus $p_{n-1}$ is the greatest of all occurrences of $C_{j}$ and $q_{n-1}$ is the least of all occurrences of $C_{j+1}$.  So $L^k(p)_{n-1} = L^k(q)_{n-1}$ and we see
\begin{align*}
L^k(p)_{n-1} &= \abs{U_0} + \cdots + \abs{U_{j}} \\
L^k(q)_{n-1} &= \abs{V_0} + \cdots + \abs{V_{j}} + 1.
\end{align*}
We will now consider the cases when $u_{n-1} \neq v_{n-1}$ and when $u_{n-1} = v_{n-1}$.

$\vspace{0.5ex}$

\noindent $\textbf{(a)}$ First, suppose $u_{n-1} \neq v_{n-1}$.  Thus we have that $u_{n-1} = 0$ and $v_{n-1} = 1$, because $\w[a+n-1] < \w[b+n-1]$ is given.  

Recall that $\abs{u}_0$ is the number of occurrences of the letter $0$ in the word $u$.  Thus $\abs{u}_0 = \abs{v}_0 + 1$, and we also have $p'_{2n-2} < p'_{2n-1}$ and $q'_{2n-2} > q'_{2n-1}$.  Thus there are exactly $\abs{v}_0$ many $j$ so that $q_{n-1} > q_j$, thus $L^k(p)_{n-1} = L^k(q)_{n-1} = \abs{v}_0 + 1 = \abs{u}_0$.  By Proposition $\ref{CalcTheFwdImage}$ we see
\begin{align*}
p'_{2n-1} &= L^k(p)_{n-1} + \abs{u}_0 = 2 \abs{u}_0 = 2 \abs{v}_0 + 2 \\
q'_{2n-1} &= L^k(q)_{n-1} + \abs{v}_0 = 2 \abs{v}_0 + 1
\end{align*}
and thus $\abs{p'_{2n-1} - q'_{2n-1}} \geq 1$.

Fix an $0 \leq i \leq n-3$ so that $u[i, i+1] = v[i, i+1] = 01$ and an $0 \leq \hat{i} \leq n-3$ so that $u[\hat{i}, \hat{i}+1] = v[\hat{i}, \hat{i}+1] = 10$.  Because $\w[a+i]$ and $\w[a+n-1]$ both have $01$ as a prefix and $p_{n-1} > p_i$, and $\w[b+\hat{i}]$ and $\w[b+n-1]$ both have $10$ as a prefix and $q_{n-1} < q_i$, we have
$$p'_{2n-2} < p'_{2i+1} < p'_{2n-1} = 2 \abs{v}_0 + 2, \hspace{3.0ex} q'_{2n-2} > q'_{2\hat{i}+1} > q'_{2n-1} = 2 \abs{v}_0 + 1 $$
and thus
$$ p'_{2n-2} < 2 \abs{v}_0 + 1 = 2 \abs{v}_0 + 2 - 1 < q'_{2n-2} - 1. $$
Therefore $\abs{p'_{2n-2} - q'_{2n-2}} \geq 2$ which satisfies the lemma.  The fact that this difference is at least 2 will be used again in Claim $\ref{RestEqualThenduISdv}$.

$\vspace{0.5ex}$

\noindent $\textbf{(b)}$ Now suppose $u = v$, so there is an $\alpha \in \{ 0, 1\}$ so that $u_{n-1} = v_{n-1} = \alpha$.

Now we investigate how the size of the $U_i$ sets are related to the size of the $V_i$ sets.  Since $u = v$, and these words have $\alpha$ as a suffix, there is some $m \geq 1$ so that $\beta \alpha^m$, where $\beta = \overline{\alpha}$, is a suffix of both $u$ and $v$.  Thus for each $0 \leq h \leq n-m-1$, there is some $i$ so $\w[a+h]$ and $\w[b+h]$ have $C_i$ as a prefix, so $h \in U_i$ and $h \in V_i$.  Moreover, this prefix is totally contained within $u$ and $v$, respectively, because $\w[a+n-m-1]$ and $\w[b+n-m-1]$ begin with the class $\beta \alpha$.  

The proofs for when $\alpha = 0$ or $\alpha = 1$ are slightly different.  We now consider the case when $\alpha = 0$, and will then give justification why the other case is also true.  Because $n-1 \in U_j$ and $n-1 \in V_{j+1}$ and $\alpha = 0$, the organization of the $U_i$ and $V_i$ sets is as follows
$$ \text{
\begin{tabular}{ l  l c l l }
 $n-m \in U_{j-m+1},$ & $n-m+1 \in U_{j-m+2},$ & $\cdots,$ & $n-2 \in U_{j-1},$ & $n-1 \in U_j$ \\
 $n-m \in V_{j-m+2},$ & $n-m+1 \in V_{j-m+3},$ & $\cdots,$ & $n-2 \in V_{j},$ & $n-1 \in V_{j+1}.$
\end{tabular}
} $$
For example, if $m = 1$ we have $n-1 \in U_j$ and $n-1 \in V_{j+1}$, so 
$$\abs{U_j} = \abs{V_j}+1 \hspace{3.0ex} \abs{U_{j+1}} = \abs{V_{j+1}}-1$$
and $\abs{U_i} = \abs{V_i}$ for all other $i$.  Since $\abs{V_j} \geq 1$ we see $\abs{U_j} \geq 2$.

If $m=2$ we have $n-2 \in U_{j-1}$, $n-2 \in V_j$, $n-1 \in U_j$, and $n-1 \in V_{j+1}$, so 
$$\abs{U_{j-1}} = \abs{V_{j-1}}+1 \hspace{3.0ex} \abs{U_j} = \abs{V_j} \hspace{3.0ex} \abs{U_{j+1}} = \abs{V_{j+1}}-1$$
and $\abs{U_i} = \abs{V_i}$ for all other $i$.  

Thus for a general $m \geq 1$,
$$\abs{U_{j-m+1}} = \abs{V_{j-m+1}}+1 \hspace{3.0ex} \abs{U_{j+1}} = \abs{V_{j+1}}-1$$
and $\abs{U_i} = \abs{V_i}$ for all other $i$.  Since $\abs{V_{j-m+1}} \geq 1$ we see $\abs{U_{j-m+1}} \geq 2$.  Each occurrence of $C_{j-m+1}$ which is contained in $u$ will have $C_j$ as a suffix, and since $n-1 \in U_j$ we have $\abs{U_j} \geq \abs{U_{j-m+1}} \geq 2$.  Thus by Proposition $\ref{CalcTheFwdImage}$ 
\begin{align*}
p'_{2n-2} &= L^k(p)_{n-1} + \abs{U_0} + \cdots + \abs{U_{j-1}} \\
q'_{2n-2} &= L^k(q)_{n-1} + \abs{V_0} + \cdots + \abs{V_{j-1}} + \abs{V_{j}} = L^k(p)_{n-1} + \abs{U_0} + \cdots + \abs{U_{j-1}} + \abs{U_{j}} - 1 \\
 &= p'_{2n-2} + \abs{U_j} - 1 \geq p'_{2n-2} + 1 \\
\\
p'_{2n-1} &= L^k(p)_{n-1} + \abs{U_0} + \cdots + \abs{U_j} \\
q'_{2n-1} &= L^k(q)_{n-1} + \abs{V_0} + \cdots + \abs{V_{j}} + \abs{V_{j+1}} = L^k(p)_{n-1} + \abs{U_0} + \cdots + \abs{U_j} + \abs{U_{j+1}} \\
 &= p'_{2n-1} + \abs{U_{j+1}} \geq p'_{2n-1} + 1
\end{align*}
and we see $\abs{p'_{2n-2} - q'_{2n-2}} \geq 1$ and $\abs{p'_{2n-1} - q'_{2n-1}} \geq 1$.

When considering the case when $\alpha = 1$ we see $u$ and $v$ will have $01^m$ as a suffix, so we find $\abs{V_{j+m}} = \abs{U_{j+m}}+1$, $\abs{V_{j}} = \abs{U_{j}}-1$, and $\abs{U_i} = \abs{V_i}$ for all other $i$.  We then find $\abs{V_j} \geq \abs{V_{j+1}} \geq \abs{V_{j+m}} \geq 2$, $q'_{2n-2} = p'_{2n-2} + \abs{V_{j+1}} - 1$, and $q'_{2n-1} = p'_{2n-1} + \abs{V_{j}}$.  Thus again, $\abs{p'_{2n-2} - q'_{2n-2}} \geq 1$ and $\abs{p'_{2n-1} - q'_{2n-1}} \geq 1$.
$\qed$
\end{proof}

The following definitions describe patterns which can occur within a set of subpermutations.

\begin{definition}
A subpermutation $p = \pi[a,a+n]$ is of $\textit{type $k$}$, for $k \geq 1$, if $p$ can be decomposed as
$$ p = ( \alpha_1 \cdots \alpha_k \lambda_1 \cdots \lambda_l \beta_1 \cdots \beta_k ) $$
where $\alpha_i = \beta_i + \varepsilon$ for each $i = 1, 2, \ldots, k$ and an $\varepsilon \in \{ -1, 1 \}$.
\end{definition}

Some examples of subpermutations of type $1$, 2, and 3 (respectively) are:
$$ (2 \hspace{.5ex} 3 \hspace{.5ex} 5 \hspace{.5ex} 4 \hspace{.5ex} 1) \hspace{4.0ex} (2 \hspace{.5ex} 5 \hspace{.5ex} 4 \hspace{.5ex} 1 \hspace{.5ex} 3 \hspace{.5ex} 6) \hspace{4.0ex} (3 \hspace{.5ex} 7 \hspace{.5ex} 5 \hspace{.5ex} 1 \hspace{.5ex} 2 \hspace{.5ex} 6 \hspace{.5ex} 4) $$

\begin{definition}
Suppose that the subpermutation $p = \pi[a,a+n]$ is of type $k$ so that for $\varepsilon \in \{-1, 1 \}$, $\alpha_i = \beta_i + \varepsilon$ for each $i = 1, 2, \ldots, k$.  If there exists a subpermutation $q = \pi[b,b+n]$ of type $k$ so that $p$ and $q$ can be decomposed as:
\begin{align*}
p &= \pi_T[a,a+n] = (\alpha_1 \cdots \alpha_k \lambda_1 \cdots \lambda_l \beta_1 \cdots \beta_k) \\ 
q &= \pi_T[b,b+n] = (\beta_1 \cdots \beta_k \lambda_1 \cdots \lambda_l \alpha_1 \cdots \alpha_k)  
\end{align*}
then $p$ and $q$ are said to be a $\textit{complementary pair of type $k$}$.  If $p$ and $q$ are a complementary pair of type $k \leq 0$ then $p = q$.
\end{definition}

The subpermutations 
$$  (2 \hspace{.5ex} 3 \hspace{.5ex} 5 \hspace{.5ex} 4 \hspace{.5ex} 1) \hspace{4.0ex} (1 \hspace{.5ex} 3 \hspace{.5ex} 5 \hspace{.5ex} 4 \hspace{.5ex} 2) $$
are a complementary pair of type 1.

\begin{lemma}
\label{NoType1PairsInDW}
For the word $\w$, let $p$, $q$, $p'$, and $q'$ be as above, then $p'$ and $q'$ are not a complementary pair of type 1.
\end{lemma}
\begin{proof}
This proof will be done by contradiction, so we assume that $p'$ and $q'$ are a complementary pair of type 1.  Thus there is an $x$, where $1 \leq x \leq 2n-1$, so $p'$ and $q'$ can be decomposed as 
\begin{align*}
p' &= \pi_{d(w)}[2a,2a+2n - 1] = ( x \hspace{0.5ex} \lambda_1 \cdots \lambda_{2n-3} \hspace{0.5ex} (x+1) ) \\ 
q' &= \pi_{d(w)}[2b,2b+2n - 1] = ( (x+1) \hspace{0.5ex} \lambda_1 \cdots \lambda_{2n-3} \hspace{0.5ex} x )  
\end{align*}
where each $\lambda_i \in \{1, 2, \ldots, 2n\}$.  

Note that for each $i$, $\lambda_i$ is not $x$ or $x+1$.  So for each $0 \leq i,j \leq n-1$ we have $p'_{2i} < p'_{2j}$ if and only if $q'_{2i} < q'_{2j}$.  Because the letter doubling map is order preserving, we then see $L^k(p)_{i} < L^k(p)_{j}$ if and only if $L^k(q)_{i} < L^k(p)_{j}$, so $L^k(p) = L^k(q)$.

Let $u = \w[a,a+n-1]$, and $v = \w[b,b+n-1]$.  Since $p'$ and $q'$ are a complementary pair they have the same form, and since $d(u)_{2n-1} = d(u)_{2n-2} = d(v)_{2n-2} = d(v)_{2n-1}$ we find that $d(u) = d(v)$ and $u = v$.  Let $\alpha \in \{ 0, 1\}$ so that $u_0 = v_0 = \alpha$.  If $d(u)_{2n-1} = d(v)_{2n-1} \neq \alpha$ we find $p'_0 < p'_{2n-1}$ if and only if $q'_0 < q'_{2n-1}$, so $u_{n-1} = v_{n-1} = \alpha$.  

Then $\alpha$ will either be 0 or 1.  We will now consider the case when $\alpha$ is 0.

%

Suppose $\alpha = 0$, so $q'_0 < q'_1$ and $q'_{2n-2} < q'_{2n-1}$ (likewise $p'_0 < p'_1$ and $p'_{2n-2} < p'_{2n-1}$).  Thus from the decomposition of $q'$ above we see 
$$ q'_{2n-2} < q'_{2n-1} < q'_0 < q'_1 $$
and by Lemma $\ref{ImageOfTheCiClasses}$ we know there is a $i \neq j$ so that $\w[b]$ has $C_j$ as a prefix and $\w[b+n-1]$ has $C_i$ as a prefix.  Since $u = v$, we know that $\w[a]$ has $C_j$ as a prefix as well.  Then from Lemma $\ref{ImageOfTheCiClasses}$ the ordering of these same terms from $p'$ must be 
$$ p'_{2n-2} < p'_0 < p'_{2n-1} < p'_1 $$
because $\alpha = 0$ and $p'_0 < p'_{2n-1}$, so both $\w[a]$ and $\w[a+n-1]$ have $C_j$ as a prefix.  

Thus $L^k(p) = L^k(q)$ and each of $\w[a+n-1]$ and $\w[b+n-1]$ have different $C_j$ classes as a prefix.  Thus we know that $p'_{2n-2} \neq q'_{2n-2}$ from Lemma $\ref{DiffCjSuffixSameLk}$ which is a contradiction to the assumption.  A similar contradiction is found if $\alpha = 1$.  In this case we see $p'_1 < p'_0 < p'_{2n-1} < p'_{2n-2}$ and $q'_1 < q'_{2n-1} < q'_0 < q'_{2n-2}$, so $\w[a]$, $\w[b]$, and $\w[b+n-1]$ each have $C_j$ as a prefix while $\w[a+n-1]$ has $C_{j+1}$ as a prefix.
$\qed$
\end{proof}

%
\begin{claim}
\label{RestEqualThenduISdv}
Suppose $f$ is a restriction map, so either $f = R$, $f = L$, or $f = M$.  If $f(p') = f(q')$ then $d(u) = d(v)$.
\end{claim}
\begin{proof}
Again we have $p = \pi_\w[a,a+n+k-1]$, $q = \pi_\w[b,b+n+k-1]$, $u = \w[a,a+n-1]$, and $v = \w[b,b+n-1]$.

%
%
$\vspace{0.5ex}$

$\bf{(a)}$ Suppose $f = R$, so $R(p') = R(q')$.  For each $0 \leq i \leq 2n-3$, $R(p')_i < R(p')_{i+1}$ is and only if $R(q')_i < R(q')_{i+1}$ and thus $d(u)_{i+1} = d(v)_{i+1}$.  Since $d(u)_0 = d(u)_1$ and $d(v)_0 = d(v)_1$, we see $d(u)_0 = d(u)_1 = d(v)_1 = d(v)_0$.  In a similar fashion we see $d(u)_{2n-1} = d(u)_{2n-2} = d(v)_{2n-2} = d(v)_{2n-1}$.  Thus $d(u) = d(v)$ and $u=v$.

%
%
$\vspace{0.5ex}$

$\bf{(b)}$ Suppose $f = L$, so $L(p') = L(q')$ and assume $d(u) \neq d(v)$.  Thus for each $0 \leq i, j \leq n-1$, $L(p')_{2i} < L(p')_{2j}$ if and only if $L(q')_{2i} < L(q')_{2j}$, so $L^k(p)_{i} < L^k(p)_{j}$ if and only if $L^k(q)_{i} < L^k(q)_{j}$ so $L^k(p) = L^k(q)$.  Thus $\w[a,a+n-2] = \w[b,b+n-2]$ and so $d(\w)[2a,2a+2n-3] = d(\w)[2b,2b+2n-3]$.  Thus $d(\w)[2a+2n-2,2a+2n-1] \neq d(\w)[2b+2n-2,2b+2n-1]$, and so $u_{n-1} \neq v_{n-1}$.  Thus $\w[a+n-1]$ and $\w[b+n-1]$ not only have different $C_j$ classes as a prefix, but they begin with different letters.  So as seen in Lemma $\ref{DiffCjSuffixSameLk}$, $\abs{p'_{2n-2} - q'_{2n-2}} \geq 2$ and thus $L(p')_{2n-2}$ and $L(p')_{2n-2}$ can not be equal which is a contradiction to the assumption.  Therefore $d(u) = d(v)$ and $u = v$.

%
%
$\vspace{0.5ex}$

$\bf{(c)}$ Suppose $f = M$, so $M(p') = M(q')$.  For $0 \leq i \leq 2n-3$, $d(u)_i = d(v)_i$ as in part $\bf{(a)}$.  Then assuming $d(u) \neq d(v)$ we find a contradiction as in part $\bf{(b)}$, so $d(u) = d(v)$.  Therefore if $M(p') = M(q')$ then $d(u) = d(v)$, and $u = v$.
$\qed$
\end{proof}

%
We are now to the main result of this section.  We show that when $\delta$ is injective we find that both of $\delta_L$ and $\delta_R$ are injective.

\begin{proposition}
\label{DubRestAreBijection}
For the word $\w$, let $p$, $q$, $p'$, and $q'$ be as above.  Then
\begin{itemize}
\item[(a)] $p' = q'$ if and only if $R(p') = R(q')$.
\item[(b)] $p' = q'$ if and only if $L(p') = L(q')$.
\end{itemize}
\end{proposition}
\begin{proof}
Again we have $p = \pi_\w[a,a+n+k-1]$, $q = \pi_\w[b,b+n+k-1]$, $u = \w[a,a+n-1]$, and $v = \w[b,b+n-1]$.  For both of these cases it should be clear that if $p' = q'$ then each of $R(p') = R(q')$ and $L(p') = L(q')$.  
%

We will again use the notation
$$ U_j = \{ \hspace{0.5ex} 0 \leq i \leq n-1 \hspace{0.5ex} | \hspace{0.5ex} \w[a+i] \text{ has } C_j \text{ as a prefix.} \hspace{0.5ex} \} $$
$$ V_j = \{ \hspace{0.5ex} 0 \leq i \leq n-1 \hspace{0.5ex} | \hspace{0.5ex} \w[b+i] \text{ has } C_j \text{ as a prefix.} \hspace{0.5ex} \} $$
and due to the length of $u$ and $v$ we know $\abs{U_j} \geq 1$ and $\abs{V_j} \geq 1$ for each $j$.

$\vspace{1.0ex}$

$\textbf{(a)}$  Suppose $p' \neq q'$, and assume $R(p') = R(q')$.  So by Claim $\ref{RestEqualThenduISdv}$ we know $d(u) = d(v)$ and $u = v$.

For each pair of real numbers $i \neq j$ where $0 \leq i,j \leq 2n-2$, 
$$R(p')_i < R(p')_j \iff R(q')_i < R(q')_j$$
and thus
$$p'_{i+1} < p'_{j+1} \iff q'_{i+1} < q'_{j+1}.$$
Since $p' \neq q'$ there must be some $1 \leq i \leq 2n-1$ so, without loss of generality, 
$$ p'_0 < p'_i \text{ and } q'_0 > q'_i. $$
There is an $\alpha \in \{ 0,1 \}$ so $d(u)_1 = d(v)_1 = \alpha$, and so $d(u)_0 = d(v)_0 = \alpha$.  If $d(u)_i = d(v)_i \neq \alpha$ we have $p'_0 < p'_i$ if and only if $q'_0 < q'_i$, which would be a contradiction.  So $d(u)_i = d(v)_i = \alpha$.

$\vspace{0.5ex}$

$\textbf{Case a.1:}$  Suppose for $1 \leq i \leq 2n-2$ we have $ p'_0 < p'_i$ and $q'_0 > q'_i$.  If $d(u)_{i+1} = d(v)_{i+1} \neq \alpha$ we have $d(u)[0,1] = \alpha \alpha$ and $d(u)[i,i+1] = \alpha \beta$, so $p'_0 < p'_i$ if and only if  $q'_0 < q'_i$, which is a contradiction, so $d(u)_{i+1} = d(v)_{i+1} = \alpha$.  Thus $d(u)[i,i+1] = d(v)[i,i+1] = \alpha \alpha$ and 
$$ p'_0 < p'_i \implies p'_1 < p'_{i+1} \implies R(p')_0 < R(p')_i $$
$$ q'_0 > q'_i \implies q'_1 > q'_{i+1} \implies R(q')_0 > R(q')_i $$
by Claim $\ref{PCClaim01}$ which contradicts the assumption.  Therefore $R(p') \neq R(q')$.


$\vspace{0.5ex}$

$\textbf{Case a.2:}$  Suppose $p'_0 < p'_{2n-1}$ and $q'_0 > q'_{2n-1}$, and for each other $i$ we have $p'_0 < p'_i$ if and only if $q'_0 < q'_i$.  So as above, we have $d(u)[0,1] = d(v)[0,1] = \alpha \alpha$ and $d(u)[2n-2,2n-1] = d(v)[2n-2,2n-1] = \alpha \alpha$.  Thus $p'_0 < p'_{2n-1}$ and $q'_0 > q'_{2n-1}$ imply
$$p'_{2n-1}-1 = R(p')_{2n-2} = R(q')_{2n-2} = q'_{2n-1}.$$

For each $1 \leq j \leq 2n-2$ we know 
$$p'_0 < p'_j \iff q'_0 < q'_j \hspace{2.0ex} \implies \hspace{2.0ex} R(p')_{j-1} = p'_j \iff R(q')_{j-1} = q'_j$$
and so $p'_j = q'_j$ for each $1 \leq j \leq 2n-2$.  So only $p'_0 \neq q'_0$ and $p'_{2n-1} \neq q'_{2n-1}$.  Since $p'_{2n-1} = q'_{2n-1} +1$, it must be 
$$p'_0  = p'_{2n-1} - 1 \text{ and } q'_0  = q'_{2n-1} + 1 .$$
Let $1 \leq x \leq 2n$ so that $p'_0 = q'_{2n-1} = x$ and $q'_0 = p'_{2n-1} = x+1$.  Thus $p'$ and $q'$ can be decomposed as $p' = ( x \hspace{0.5ex} \lambda_1 \cdots \lambda_{2n-3} \hspace{0.5ex} (x+1) )$ and $q' = ( (x+1) \hspace{0.5ex} \lambda_1 \cdots \lambda_{2n-3} \hspace{0.5ex} x )$.  So we have that $p'$ and $q'$ are a complementary pair of type 1, which is a contradiction by Lemma $\ref{NoType1PairsInDW}$.  Thus $R(p') \neq R(q')$.

Therefore $p' = q'$ if and only if $R(p') = R(q')$.

$\vspace{1.0ex}$

$\textbf{(b)}$  Suppose $p' \neq q'$, and assume $L(p') = L(q')$.  So by Claim $\ref{RestEqualThenduISdv}$ we know $d(u) = d(v)$ and $u = v$.

For each pair of real numbers $i \neq j$ where $0 \leq i,j \leq 2n-2$, 
$$L(p')_i < L(p')_j \iff L(q')_i < L(q')_j$$
and thus
$$p'_{i} < p'_{j} \iff q'_{i} < q'_{j}.$$
As in Claim $\ref{RestEqualThenduISdv}$ we can also see that $L^k(p) = L^k(q)$.

Since $p' \neq q'$ there must be some $0 \leq i \leq 2n-2$ so, without loss of generality, 
$$ p'_{2n-1} < p'_i \text{ and } q'_{2n-1} > q'_i. $$
There is an $\alpha \in \{ 0,1 \}$ so $d(u)_{2n-2} = d(v)_{2n-2} = \alpha$, so $d(u)[2n-2, 2n-1] = d(v)_[2n-2, 2n-1] = \alpha \alpha$.  If $d(u)_i = d(v)_i \neq \alpha$ we have $p'_{2n-2} < p'_i$ if and only if $q'_{2n-2} < q'_i$, which would be a contradiction, so $d(u)_i = d(v)_i = \alpha$.  It should be noted that $i \neq 2n-2$, because $d(u)_{2n-2} = d(v)_{2n-2} = \alpha$ so $p'_{2n-2} < p'_{2n-1}$ if and only if $q'_{2n-2} < q'_{2n-1}$.  

$\vspace{0.5ex}$

$\textbf{Case b.1:}$  Suppose for $1 \leq i \leq 2n-2$ we have $ p'_{2n-1} < p'_i$ and $q'_{2n-1} > q'_i$.  If $d(u)_{i-1} = d(v)_{i-1} = \alpha$ we have $d(u)[2n-2,2n-1] = \alpha \alpha$ and $d(u)[i-1,i] = \alpha \alpha$, so 
$$ p'_{2n-1} < p'_i \implies p'_{2n-2} < p'_{i-1} $$
$$ q'_{2n-1} > q'_i \implies q'_{2n-1} > q'_{i-1} $$
which contradicts the assumption.  So $d(u)_{i-1} = d(v)_{i-1} \neq \alpha$, say $d(u)_{i-1} = d(v)_{i-1}  = \beta$.  Thus $d(u)[i-1,i+1] = \beta \alpha \alpha$ and $i$ is an even number, so rather than using $i$ we will use $2c$.  

Because $d(u)[2n-2,2n-1] = d(v)[2n-2,2n-1] = \alpha \alpha$ we know $p'_{2n-2} < p'_{2n-1}$ if and only if $q'_{2n-2} < q'_{2n-1}$, and thus $p'_{2n-2} = L(p')_{2n-2}$ if and only if $L(q')_{2n-2} = q'_{2n-2}$, so $p'_{2n-2} = q'_{2n-2}$ because $L(p') = L(q')$.

If $\alpha = 0$ we find 
$$p'_{2n-2} < p'_{2n-1} < p'_{2c} < p'_{2c+1} \hspace{2.0ex} \text{ and } \hspace{2.0ex} q'_{2n-2} < q'_{2c} < q'_{2n-1} < q'_{2c+1}, $$
and if $\alpha = 1$ we find 
$$q'_{2c+1} < q'_{2c} < q'_{2n-1} < q'_{2n-2} \hspace{2.0ex} \text{ and } \hspace{2.0ex} p'_{2c+1} < p'_{2n-1} < p'_{2c} < p'_{2n-2}.$$
In either case, $\w[a+n-1]$ and $\w[b+n-1]$ each have a different $C_j$ class as a prefix.  So by Lemma $\ref{DiffCjSuffixSameLk}$ we have $\abs{p'_{2n-2} - q'_{2n-2}} \geq 1$, which is a contradiction to the assumption, and $L(p') \neq L(q')$.

$\vspace{0.5ex}$

$\textbf{Case b.2:}$  Suppose $p'_{2n-1} < p'_0$ and $q'_{2n-1} > q'_0$, and for each other $i$ we have $p'_0 < p'_i$ if and only if $q'_0 < q'_i$.  So as above we have $d(u)[0,1] = d(v)[0,1] = \alpha \alpha$ and $d(u)[2n-2,2n-1] = d(v)[2n-2,2n-1] = \alpha \alpha$.  Thus $p'_{2n-1} < p'_0$ and $q'_{2n-1} > q'_0$ imply $p'_0 = L(p')_0 = L(q')_0 = q'_0-1$.  For each $1 \leq j \leq 2n-2$ we know the following
$$p'_{2n-1} < p'_j \iff q'_{2n-1} < q'_j, \hspace{4.0ex} L(p')_{j} = p'_j \iff L(q')_{j} = q'_j,$$
and thus $p'_j = q'_j$ for each $1 \leq j \leq 2n-2$, because $L(p')_{j} = L(q')_{j}$.  

So only $p'_0 \neq q'_0$ and $p'_{2n-1} \neq q'_{2n-1}$.  Since $q'_0 = p'_0 +1$, it must be $p'_{2n-1}  = p'_0 - 1 \text{ and } q'_{2n-1}  = q'_0 + 1$.  Let $1 \leq x \leq 2n$ so that $p'_{2n-1} = q'_0 = x$ and $q'_{2n-1} = p'_0 = x+1$.  So once again we have that $p'$ and $q'$ are a complementary pair of type 1, which is a contradiction by Lemma $\ref{NoType1PairsInDW}$.  Thus $L(p') \neq L(q')$.

Therefore $p' = q'$ if and only if $L(p') = L(q')$.
$\qed$
\end{proof}

Therefore when $\delta$ is injective, $\delta_R$ and $\delta_L$ are both injective as well.  A troubling fact is the map $\delta$ being injective does not imply $\delta_M$ is injective.  As will be shown for the Thue-Morse word $T$, there are cases of distinct subpermutations $p$ and $q$ where $\delta(p) \neq \delta(q)$ but $\delta(p)_M = \delta_M(q)$.  The following sections deal with some different words and we will show when $\delta$ and $\delta_M$ are injective, but these proofs will use special properties of the words considered.

\section{Sturmian Words}
\label{SecSturmianWords}

In this section we will investigate the permutation complexity of Sturmian words under the doubling map.  An infinite word $s$ is a $\textit{Sturmian word}$ if for each $n \geq 0$, $s$ has exactly $n+1$ distinct factors of length $n$, or $\rho_s(n) = n+1$ (the only factor of length $n=0$ being the empty-word).  Thus since $\rho_s(1) = 2$, it should be clear that Sturmian words are binary words.  The class of Sturmian words have been a topic of much study (see $\cite{Berstel96, CovHed73, AlgCombOnWords}$).  An equivalent definition for Sturmian words is that they are the class of aperiodic balanced binary words.  A word is $\textit{balanced}$ if for all factors $u$ and $v$ with $\abs{u} = \abs{v}$, $\abs{ \abs{u}_a - \abs{v}_a } \leq 1$ for each $a$ in the alphabet.

First we will show when the map $\delta$ is applied to permutations from a Sturmian word, $\delta$ is injective and thus a bijection.  Then we show the maps $\delta_R$, $\delta_L$, and $\delta_M$ are injective as well and thus also bijections.  First we look at the permutation complexity of the Sturmian words which has been calculated.
\begin{lemma}
\emph{($\cite{Makar09}$)}
\label{PermComp02}
Let $s$ be a Sturmian word.  For natural numbers $a_1$ and $a_2$ we have $\pi_{s}[a_1,a_1 + n + 1] = \pi_{s}[a_2,a_2 + n+1]$ if and only if $s[a_1, a_1+n] = s[a_2, a_2+n]$.
\end{lemma}

\begin{theorem}
\emph{($\cite{Makar09}$)}
\label{PermCompOfSturmian}
Let $s$ be a Sturmian word.  For each $n \geq 2$, $\tau_s(n) = n$ .
\end{theorem}
If $s$ is a Sturmian word, then $d(s)$ is not Sturmian.  The word $d(s)$ will contain both $00$ and $11$ as factors and is not balanced.  Thus we know $\tau_{d(s)}(n) > n$.

Fix a Sturmian word $s$ over $\{ 0, 1 \}$.  Since $s$ is balanced, there is some $k >0$ so that for $\alpha, \beta \in \{ 0, 1 \}$, with $\alpha \neq \beta$, every $\alpha$ is followed by either $k$ or $k-1$ $\beta$'s.  So consecutive $\alpha$'s will look like either $\alpha \beta^k \alpha$ or $\alpha \beta^{k-1} \alpha$.  For example consider the Fibonacci word, $t= 01001010010010100101\cdots$, where consecutive $1$'s look like either $1001$ or $101$.

Let $d(s)$ be the image of $s$ under the doubling map.  Let $\pi_s$ be the infinite permutation associated to $s$, and $\pi_{d(s)}$ be the infinite permutation associated to $d(s)$.  We will now calculate the permutation complexity of $d(s)$.  By Lemma $\ref{PermCompOfCompliment}$ we may assume there is a natural number $k > 0$ so that each $1$ is followed by either $0^k1$ or $0^{k-1}1$, because $d(s)$ and $d(\overline{s})$ have the same permutation complexity.  There will be $k+1$ classes of factors of $s$, which are $C_0 = 0^{k}$, $C_1 = 0^{k-1}1$, $\cdots$, $C_{k-1} = 01$, $C_{k} = 10$.  For each $i \in \N$, $s[i] = s_i s_{i+1} \cdots$ will have exactly one the above classes of words as a prefix.  Since Sturmian words are uniformly recurrent ($\cite{CovHed73}$), there is an $N \in \N$ so that each factor of $s$ of length $n \geq N$ will contain each of $C_0$, $C_1$, $\ldots$, $C_k$.

Let $u = s[a,a+n-1]$ and $v = s[b,b+n-1]$, $a \neq b$, be factors of $s$ of length $n \geq N$, so $C_j$ is a factor of both $u$ and $v$ for each $0 \leq j \leq k$.  For $0 \leq j \leq k$ define 
$$ U_j = \{ \hspace{0.5ex} 0 \leq i \leq n-1 \hspace{0.5ex} | \hspace{0.5ex} T[a+i] \text{ has } C_j \text{ as a prefix.} \hspace{0.5ex} \} $$
$$ V_j = \{ \hspace{0.5ex} 0 \leq i \leq n-1 \hspace{0.5ex} | \hspace{0.5ex} T[b+i] \text{ has } C_j \text{ as a prefix.} \hspace{0.5ex} \} $$
and $\abs{U_0} + \abs{U_1} + \cdots + \abs{U_k} = \abs{V_0} + \abs{V_1} + \cdots + \abs{V_k} = n$.  Since $\abs{u} = \abs{v} \geq N$ we know for each $j$ there is an occurrence of $C_j$ in both $u$ and $v$ so $\abs{U_j} \geq 1$ and $\abs{V_j} \geq 1$.  Let $p = \pi_s[a,a+n+k-1]$ and $q = \pi_s[b,b+n+k-1]$ be subpermutations of $\pi_s$.  Then define subpermutations $\delta(p) = p' = \pi_{d(s)}[2a,2a+2n-1]$ and $\delta(q) = q' = \pi_{d(s)}[2b,2b+2n-1]$ as in Proposition $\ref{CalcTheFwdImage}$.  The following lemma concerns the relationship of $p$ and $q$ to $p'$ and $q'$.

\begin{lemma}
\label{ForSturm_pISqIFFppISqp}
For the Sturmian word $s$, let $p$, $q$, $p'$, and $q'$ be as above.  Then $p = q$ if and only if $p' = q'$.
\end{lemma}
\begin{proof}
If $p=q$, then it follows from Lemma $\ref{pISqTHENppISqp}$ that $p' = q'$.

Then suppose that $p \neq q$.  Thus $p$ and $q$ have a different form by Lemma $\ref{PermComp02}$.  Thus there is an $0 \leq i \leq n+k-2$ so that, without loss of generality, $p_i < p_{i+1}$ and $q_i > q_{i+1}$.  We will look at the least $i$ where this happens and it will be handled in two cases.  First when $0 \leq i \leq n-1$, and then when $n \leq i \leq n+k-2$.


$\textbf{Case a:}$  Suppose $0 \leq i \leq n-1$ is the least $i$ where $p_i < p_{i+1}$ and $q_i > q_{i+1}$.  Then $p' \neq q'$ follows from Corollary $\ref{Corr_uNEQv}$.


$\textbf{Case b:}$  Suppose $n \leq i \leq n+k-2$ is the least $i$ where $p_i < p_{i+1}$ and $q_i > q_{i+1}$.  Thus we know $u = s[a,a+n-1] = s[b,b+n-1] = v$, and so $L^k(p) = L^k(q)$ by Lemma $\ref{PermComp02}$.  

If $u_{n-1} = v_{n-1} = 1$, then both $u$ and $v$ are followed by $0^{k-1}$ so $s[a,a+n+k-2] = s[b,b+n+k-2] = u0^{k-1}$ and $p = q$ contradicting the assumption.  Thus $u_{n-1} = v_{n-1} = 0$, and letting $m = i-n+1$ 
$$ s[a] = u0^{m}1 \cdots, \hspace{4.0ex} s[b] = u0^{m-1}1 \cdots. $$
Thus we can see that $s[a+n-1]$ and $s[b+n-1]$ each have a different $C_j$ class as a prefix.  So by Lemma $\ref{DiffCjSuffixSameLk}$, $\abs{p'_{2n-2} - q'_{2n-2}} \geq 1$ and $\abs{p'_{2n-1} - q'_{2n-1}} \geq 1$ and thus $p' \neq q'$.

Therefore $p = q$ if and only if $p = q$.
$\qed$
\end{proof}
Thus the map $\delta$ is injective when applied to permutations associated with a Sturmian word, and is therefore bijective.  When Lemma $\ref{ForSturm_pISqIFFppISqp}$ is used with Proposition $\ref{DubRestAreBijection}$ we see the maps $\delta_L$ and $\delta_R$ are also injective, and thus are bijections.  We will now show the map $\delta_M$ is also injective when applied to permutations associated with a Sturmian word.


\begin{lemma}
\label{SturmianDubRestAreBijection}
For the Sturmian word $s$, let $p$, $q$, $p'$, and $q'$ be as above.  Then $p' = q'$ if and only if $M(p') = M(q')$.
\end{lemma}
\begin{proof}
It should be clear that if $p' = q'$ then $M(p') = M(q')$.

Suppose $p' \neq q'$, and assume $M(p') = M(q')$.  For each pair of real numbers $i \neq j$ where $0 \leq i,j \leq 2n-3$, 
$$M(p')_i < M(p')_j \iff M(q')_i < M(q')_j \implies p'_{i+1} < p'_{j+1} \iff q'_{i+1} < q'_{j+1}. $$
From Claim $\ref{RestEqualThenduISdv}$ we know $d(u) = d(v)$ and $u=v$, so $L^k(p) = L^k(q)$ because they have the same form.  

From Proposition $\ref{DubRestAreBijection}$ we know $R(p') \neq R(q')$ and $L(p') \neq L(q')$ because $p' \neq q'$, but 
$$R(L(p')) = M(p') = M(q') = R(L(q')).$$
Thus there is an $1 \leq i \leq 2n-2$ so that $L(p')_0 < L(p')_i$ and $L(q')_0 > L(q')_i$.

If $1 \leq i \leq 2n-3$, we find a contradiction in the same fashion as in Proposition $\ref{DubRestAreBijection}$, case $\bf{(a.1)}$.  Thus we can assume that $i = 2n-2$ is the only $i$ so that $L(p')_0 < L(p')_i$ and $L(q')_0 > L(q')_i$.  Thus 
$$ L(p')_0 < L(p')_{2n-2} \implies p'_0 < p'_{2n-2} \implies p_0 < p_{n-1} \implies L^k(p)_0 < L^k(p)_{n-1} $$
$$ L(q')_0 > L(q')_{2n-2} \implies q'_0 > q'_{2n-2} \implies q_0 > q_{n-1} \implies L^k(q)_0 > L^k(q)_{n-1}, $$
and $L^k(p) \neq L^k(q)$, and so by Lemma $\ref{PermComp02}$ we see $u \neq v$ and $d(u) \neq d(v)$ which is a contradiction to the assumption.  Therefore $M(p') \neq M(q')$.

Therefore $p' = q'$ if and only if $M(p') = M(q')$.
$\qed$
\end{proof}

Thus we see, for a Sturmian word $s$, 
$$ p = q \iff \delta(p) = \delta(q) \iff \delta_M(p) = \delta_M(q) $$
and thus the map  $\delta_M$ is also injective, and thus is a bijection.  The following theorem will give the permutation complexity of the image of a Sturmian word under the letter doubling map.

\begin{theorem}
\label{PermCompOfDS}
Let $s$ be a Sturmian word over $\scr{A}$, where for $\alpha, \beta \in \scr{A}$, $\alpha \neq \beta$, there are strings of either $k$ or $k-1$ $\alpha$ between each $\beta$.  There is an $N$ so that each factor of $s$ of length at least $N$ will contain each of $\alpha^k$, $\alpha^{k-1}\beta$, \ldots, $\alpha \beta$, $\beta$.  For each $n \geq 2N$ the permutation complexity of $d(s)$ is
$$ \tau_{d(s)}(n) = n+2k+1 $$
\end{theorem}
\begin{proof}
Let $s$ be a Sturmian word as in the hypothesis, and let $n \geq 2N$.  Then there is $m \geq N$ so that either $n = 2m$ or $n = 2m-1$.  Since $s$ is Sturmian, each of $\delta$, $\delta_L$, $\delta_R$, and $\delta_M$ 
are bijections, and so

$$ \text{
\begin{tabular}{ l  l }
 $\abs{\mathrm{Perm}^{d(s)}_{ev}(2m-1)} = \abs{\mathrm{Perm}^s(m+k)}$, & $\abs{\mathrm{Perm}^{d(s)}_{ev}(2m)} = \abs{\mathrm{Perm}^s(m+k)}$, \\ \\
 $\abs{\mathrm{Perm}^{d(s)}_{odd}(2m-1)} = \abs{\mathrm{Perm}^s(m+k)}$, & $\abs{\mathrm{Perm}^{d(s)}_{odd}(2m)} = \abs{\mathrm{Perm}^s(m+k+1)}.$
\end{tabular}
} $$
%

Thus 
\begin{align*}
\tau_{d(s)}(2m-1) &= \abs{\mathrm{Perm}^{d(s)}(2m-1)} = \abs{\mathrm{Perm}^{d(s)}_{ev}(2m-1)} + \abs{\mathrm{Perm}^{d(s)}_{odd}(2m-1)} \\
& = (m+k) + (m+k) = (2m-1) + 2k + 1 \\
\\
\tau_{d(s)}(2m) &= \abs{\mathrm{Perm}^{d(s)}(2m)} = \abs{\mathrm{Perm}^{d(s)}_{ev}(2m)} + \abs{\mathrm{Perm}^{d(s)}_{odd}(2m)} \\
& = (m+k) + (m+k+1) = 2m + 2k + 1 
\end{align*}
Therefore for either $n = 2m$ or $n = 2m-1$, $\tau_{d(s)}(n) = n + 2k + 1$.
$\qed$
\end{proof}

\section{Thue-Morse Word}
\label{SecThueMorseWord}

In this section we will investigate the permutation complexity of $d(T)$, the image of the Thue-Morse word, $T$, under the doubling map, $d$.  The Thue-Morse word is:
$$ T = 0110 1001 1001 0110 1001 0110 0110 1001 \cdots,$$
and the Thue-Morse morphism is:
$$\mu_T:0 \rightarrow 01, \hspace{1.5ex} 1 \rightarrow 10. $$
This word was introduced by Axel Thue in his studies of repetitions in words ($\cite{Thue12}$).  For a more in depth look at further properties, independent discoveries, and applications of the Thue-Morse word see $\cite{AllSha99}$.  

A nice property of the factors of $T$ is that any factor of length 5 or greater contains either $00$ or $11$.  Another interesting property is that for any $i \in \N$, $T[2i,2i+1]$ will be either 01 or 10.  Thus any occurrence of $00$ or $11$ must be a factor of the form $T[2i+1,2i+2]$ for some $i \in \N$.  Therefore any factors $T[2i,2i+n]$ and $T[2j+1,2j+1+n]$ where $n \geq 4$ cannot be equal based on the location of the factors $00$ or $11$.

The factor complexity of the Thue-Morse word was computed independently by two groups in 1989 ($\cite{Brlek89}$ and $\cite{LucaVarr89}$).  The calculation of the permutation complexity of $d(T)$ will use the formula for the factor complexity of $T$.  We will use the formula calculated by S. Brlek.

\begin{proposition}
\emph{($\cite{Brlek89}$)}
\label{TMFactComplexity}
For $n \geq 3$, the function $\rho_T(n)$ is given by 
$$ \rho_T(n) = \begin{cases} 
6 \cdot 2^{r-1} + 4p & 0 < p \leq 2^{r-1} \\
8 \cdot 2^{r-1} + 2p & 2^{r-1} < p \leq 2^{r}
\end{cases} $$
where $r$ and $p$ are uniquely determined by the equation 
$$ n = 2^r + p + 1, \hspace{4.0ex} 0 < p \leq 2^r $$
\end{proposition}

Let $\pi_T$ be the infinite permutation associated to the Thue-Morse word $T$.  In $\cite{Widmer10}$, the permutation complexity of $T$ was calculated.  
\begin{theorem}
\emph{($\cite{Widmer10}$)}
\label{PermCompIsTheFormula}
For any $n \geq 6$, where $n = 2^r + p$ with $0 < p \leq 2^r$,
$$ \tau_T(n) = 2(2^{r+1}+p-2).$$
\end{theorem}

We will now investigate the permutation complexity of $d(T)$.  To begin, we consider complementary pairs which occur in $\pi_T$.

\begin{theorem}
\emph{($\cite{Widmer10}$)}
\label{SameFormIFFCompPair}
Let $p$ and $q$ be distinct subpermutations of $\pi_T$.  Then $p$ and $q$ have the same form if and only if $p$ and $q$ are a complementary pair of type $k$, for some $k \geq 1$.
\end{theorem}

The left and right restrictions preserve complementary pairs of type $k \geq 2$, and middle restrictions preserve complementary pairs of type $k \geq 3$.  Proposition $\ref{ImageOfTypeK}$ follows directly from $\cite{Widmer10}$, Proposition 4.1.  We then see when complementary pairs of type $k$ can occur, for each $k \geq 0$.
\begin{proposition}
\emph{($\cite{Widmer10}$)}
\label{ImageOfTypeK}
Suppose $p = \pi_T[a,a+n]$ and $q = \pi_T[b,b+n]$ are a complementary pair of type $k \geq 1$.
\begin{itemize}
\item[(a)] $L(p)$ and $L(q)$ are a complementary pair of type $k-1$.
\item[(b)] $R(p)$ and $R(q)$ are a complementary pair of type $k-1$.
\item[(c)] $M(p)$ and $M(q)$ are a complementary pair of type $k-2$.
\end{itemize}
\end{proposition}

\begin{proposition}
\emph{($\cite{Widmer10}$)}
\label{LengthOfAlphaForAGBForN}
Let $n > 4$ be a natural number and let $p$ and $q$ be subpermutations of $\pi_T$ of length $n+1$ with the same form.  There exist $r$ and $c$ so that $n =2^r+c$, where $0 \leq c < 2^r$.  
\begin{itemize}
\item[(a)] If $0 \leq c < 2^{r-1}+1$, then either $p = q$ or $p$ and $q$ are a complementary pair of type $c+1$.
\item[(b)] If $2^{r-1}+1 \leq c < 2^r$, then $p = q$.
\end{itemize}
\end{proposition}
Thus only subpermutations of length $2^r+1$, for some $r \geq 1$, can be complementary pair of type 1, and only subpermutations of length $2^r+2$, for some $r \geq 1$, can be complementary pair of type 2.

Now to calculate the permutation complexity of $d(T)$ we need to identify the classes of factors of $T$ with blocks of the same letter.  Since $T$ is overlap-free, and thus cube free, we can identify the $4$ classes of factors of $T$, which are $C_0 = 00$, $C_1 = 01$, $C_2 = 10$, and $C_3 = 11$.  For each $i \in \N$, $T[i] = T_i T_{i+1} \cdots$ will have exactly one the above classes of words as a prefix.  Since the Thue-Morse word is uniformly recurrent ($\cite{AllSha99}$), there is an $N \in \N$ so that each factor of $T$ of length $n \geq N$ will contain each of $C_0$, $C_1$, $C_2$, and $C_3$.  It is readily verified that any factor of length $n \geq 9$ will contain these 4 classes of words.

Let $u = T[a,a+n-1]$ and $v = T[b,b+n-1]$, $a \neq b$, be factors of $T$ of length $n \geq 9$, so $C_j$ is a factor of both $u$ and $v$ for each $0 \leq j \leq 3$.  Let $p = \pi_T[a,a+n+1]$ and $q = \pi_T[b,b+n+1]$ be subpermutations of $\pi_T$.  Then define subpermutations $\delta(p) = p' = \pi_{d(T)}[2a,2a+2n-1]$ and $\delta(q) = q' = \pi_{d(T)}[2b,2b+2n-1]$ as in Proposition $\ref{CalcTheFwdImage}$, with $k=2$. The following lemma concerns the relationship of $p$ and $q$ to $p'$ and $q'$.

\begin{lemma}
\label{TM_pISq_DiffForm}
Let $p$ and $q$ be subpermutations of length $n+2$ of $\pi_T$, with $n \geq 9$, and let $p' = \delta(p)$ and $q' = \delta(p)$.
\begin{itemize}
\item[(a)] If $n \neq 2^r-1$ or $2^r$ for any $r \geq 3$, $p = q$ if and only if $p' = q'$.
\item[(b)] If $n = 2^r-1$ or $2^r$ for some $r \geq 3$, $p$ and $q$ have the same form if and only if $p' = q'$.
\end{itemize}
\end{lemma}
\begin{proof}
Let $p = \pi_T[a,a+n+1]$ and $q = \pi_T[b,b+n+1]$, $a \neq b$, be subpermutations of $\pi_T$ of length $n+2$, with $n \geq 9$, and let $u = T[a,a+n-1]$ and $v = T[b,b+n-1]$.  Since the length of $u$ and $v$ is at least 9, each of $C_0$, $C_1$, $C_2$, and $C_3$ occurs in both of $u$ and $v$.  Then let $p' = \delta(p) = \pi_{d(T)}[2a,2a+2n-1]$ and $q' = \delta(q) = \pi_{d(T)}[2b,2b+2n-1]$ as in Proposition $\ref{CalcTheFwdImage}$.

$\vspace{1.0ex}$

$\textbf{(a)}$  Suppose $n \neq 2^r-1$ or $2^r$ for any $r \geq 3$.  If $p = q$ then $p' = q'$ by Lemma $\ref{pISqTHENppISqp}$.

Suppose $p \neq q$.  Then either $p$ and $q$ have the same form, or they do not have the same form.  These cases will be handled independently.

$\textbf{Case (a.1)}$  Suppose $p$ and $q$ have the same form.  Since $n \neq 2^r-1$ or $2^r$ and $p$ and $q$ have the same form, $p$ and $q$ are a complementary pair of type $k \geq 3$, by Theorem $\ref{SameFormIFFCompPair}$ and Proposition $\ref{ImageOfTypeK}$.  Thus $L^2(p)$ and $L^2(q)$ are a complementary pair of type $k-2$, where $k-2 \geq 1$, and so $L^2(p) \neq L^2(q)$.  Therefore $p' \neq q'$, by Corollary $\ref{Corr_uNEQv}$.

$\textbf{Case (a.2)}$  Suppose $p$ and $q$ do not have the same form.  Thus there is an $0 \leq i \leq n$ so that, without loss of generality, $p_i < p_{i+1}$ and $q_i > q_{i+1}$.  We may say $i = n$ is the only $i$ so that $p_i < p_{i+1}$ and $q_i > q_{i+1}$, because if there is an $0 \leq i \leq n-1$ so $p_i < p_{i+1}$ and $q_i > q_{i+1}$ then $u \neq v$, and $p' \neq q'$ by Corollary $\ref{Corr_uNEQv}$.  We may also say $L^2(p) = L^2(q)$, because if $L^2(p) \neq L^2(q)$ then $p' \neq q'$ be Corollary $\ref{Corr_uNEQv}$.  Thus $u = v$ and $p_n < p_{n+1}$ and $q_n > q_{n+1}$ and we see 
$$ T[a] = u0 \cdots \hspace{3.0ex} T[b] = v1\cdots = u1\cdots. $$
Thus $T[a+n-1]$ and $T[b+n-1]$ each have a different $C_j$ class as a prefix and by Lemma $\ref{DiffCjSuffixSameLk}$, $\abs{p'_{2n-2} - q'_{2n-2}} \geq 1$ and $\abs{p'_{2n-1} - q'_{2n-1}} \geq 1$ so $p' \neq q'$.

Therefore if $p$ and $q$ are subpermutations of $\pi_T$ of length $n+2$, with $n \neq 2^r-1$ or $2^r$ for any $r \geq 3$, $p = q$ if and only if $p' = q'$.

$\vspace{1.0ex}$

$\textbf{(b)}$  Suppose $n = 2^r-1$ or $2^r$ for some $r \geq 3$.  If $p$ and $q$ do not have the same form, there is an $0 \leq i \leq n$ so that, without loss of generality, $p_i < p_{i+1}$ and $q_i > q_{i+1}$ and $p \neq q$.  Thus $p$ and $q$ are as in Case $(a.2)$, and $p' \neq q'$.

Suppose $p$ and $q$ have the same form, so for each $0 \leq i \leq n-1$, there is some $j$ so that both $T[a+i]$ and $T[b+i]$ have $C_j$ as a prefix.  We can say $p \neq q$, because if $p = q$ then $p' = q'$ by Lemma $\ref{pISqTHENppISqp}$.  By Theorem $\ref{SameFormIFFCompPair}$ and Proposition $\ref{LengthOfAlphaForAGBForN}$, $p$ and $q$ are a complementary pair of type $1$ or $2$ and $L^2(p) = L^2(q)$ by Proposition $\ref{ImageOfTypeK}$.  So by Corollary $\ref{Cor_pNEQq_SameFormAndSameRest}$, $p' = q'$. 

Therefore if $p$ and $q$ are subpermutations of $\pi_T$ of length $n+2$, with $n = 2^r-1$ or $2^r$ for some $r \geq 3$, $p$ and $q$ have the same form if and only if $p' = q'$.
$\qed$
\end{proof}

Thus, for $n \geq 9$, the maps $\delta$, $\delta_L$, and $\delta_R$ 
when applied to permutations associated with the Thue-Morse word are injective when $n \neq 2^r-1$ or $2^r$ for any $r \geq 3$, so $\abs{\mathrm{Perm}^{d(T)}_{ev}(2n)} = \abs{\mathrm{Perm}^T(n+2)}$, $\abs{\mathrm{Perm}^{d(T)}_{ev}(2n-1)} = \abs{\mathrm{Perm}^T(n+2)}$, and $\abs{\mathrm{Perm}^{d(T)}_{odd}(2n-1)} = \abs{\mathrm{Perm}^T(n+2)}$.

When $n = 2^r-1$ or $2^r$ for some $r \geq 3$ the maps $\delta$, $\delta_R$, and $\delta_L$ are surjective, but not injective because complementary pairs of type 1 or 2 will give the same subpermutation under $\delta$.  In this case, if $p$ and $q$ are subpermutations of $\pi_T$ of length $n+2$, where $p$ has form $u$ and $q$ has form $v$, $\abs{u} = \abs{v} = n+1$, $\delta(p) = \delta(q)$ if and only if $u = v$.  Thus with Proposition $\ref{DubRestAreBijection}$ we see $\delta_L(p) = \delta_L(q)$ and $\delta_R(p) = \delta_R(q)$ if and only if $u = v$.  Thus the number of subpermutations of $\pi_{d(T)}$ for these lengths are determined by the number of factors of $T$, so $\abs{\mathrm{Perm}^{d(T)}_{ev}(2n)} = \abs{\scr{F}_T(n+1)}$, $\abs{\mathrm{Perm}^{d(T)}_{ev}(2n-1)} = \abs{\scr{F}_T(n+1)}$, and $\abs{\mathrm{Perm}^{d(T)}_{odd}(2n-1)} = \abs{\scr{F}_T(n+1)}$.

The following lemma shows when the map $\delta_M$ is injective when applied to permutations associated with the Thue-Morse word.
\begin{lemma}
\label{ThueMorseDubRestAreBijection}
For the Thue-Morse word $T$, let $p$, $q$, $p'$, and $q'$ be as above.  Then
\begin{itemize}
\item[(a)] If $n \neq 2^r-1$, $2^r$, or $2^r+1$ for any $r \geq 3$, $p' = q'$ if and only if $M(p') = M(q')$.
\item[(b)] If $n = 2^r-1$, $2^r$, or $2^r+1$ for some $r \geq 3$, $p$ and $q$ have the same form if and only if $M(p') = M(q')$.
\end{itemize}
\end{lemma}
\begin{proof}

Let $p = \pi_T[a,a+n+1]$ and $q = \pi_T[b,b+n+1]$ be subpermutations of $\pi_T$ of length $n+2 \geq 11$, and $p' = \delta(p)$ and $q' = \delta(q)$ as in Proposition $\ref{CalcTheFwdImage}$.  Let $u = T[a,a+n-1]$ and $v = T[b,b+n-1]$.  It should be clear for either case that if $p' = q'$ then $M(p') = M(q')$.  

We will again use the notation
$$ U_j = \{ \hspace{0.5ex} 0 \leq i \leq n-1 \hspace{0.5ex} | \hspace{0.5ex} T[a+i] \text{ has } C_j \text{ as a prefix.} \hspace{0.5ex} \} $$
$$ V_j = \{ \hspace{0.5ex} 0 \leq i \leq n-1 \hspace{0.5ex} | \hspace{0.5ex} T[b+i] \text{ has } C_j \text{ as a prefix.} \hspace{0.5ex} \} $$
and due to the length of $u$ and $v$ we know $\abs{U_j} \geq 1$ and $\abs{V_j} \geq 1$ for each $j$.

$\vspace{0.5ex}$

$\textbf{(a)}$  Let $n \neq 2^r-1$, $2^r$, or $2^r+1$ for any $r \geq 3$.  It should be clear that if $p' = q'$ then $M(p') = M(q')$.

Suppose $p' \neq q'$, so $p \neq q$ by Lemma $\ref{TM_pISq_DiffForm}$, and assume $M(p') = M(q')$.  For each pair of real numbers $i \neq j$ where $0 \leq i,j \leq 2n-3$, 
$$M(p')_i < M(p')_j \iff M(q')_i < M(q')_j \implies p'_{i+1} < p'_{j+1} \iff q'_{i+1} < q'_{j+1}. $$
Since $M(p') = M(q')$ then $d(u) = d(v)$, by Claim $\ref{RestEqualThenduISdv}$, and so $u = v$.  

$\textbf{Case (a.1)}$  Suppose $p$ and $q$ have the same form.  By Theorem $\ref{SameFormIFFCompPair}$ and Proposition $\ref{LengthOfAlphaForAGBForN}$, $p$ and $q$ are a complementary pair of type $k \geq 4$.  By Proposition $\ref{ImageOfTypeK}$, $L^2(p)$ and $L^2(q)$ are a complementary pair of type $k-2 \geq 2$.  Thus, without loss of generality, $L^2(p)_{k-2-1}+1 = L^2(p)_{n-1}$ and $L^2(q)_{n-1}+1 = L^2(q)_{k-2-1}$.  Thus $L^2(p)_{k-3} < L^2(p)_{n-1}$ and $L^2(q)_{k-3} > L^2(q)_{n-1}$, so $p'_{2k-6} < p'_{2n-2}$ and $q'_{2k-6} > q'_{2n-2}$.  Thus $M(p')_{2k-5} < M(p')_{2n-3}$ and $M(q')_{2k-5} > M(q')_{2n-3}$ so $M(p') \neq M(q')$ which is a contradiction.

$\textbf{Case (a.2)}$  Suppose $p$ and $q$ do not have the same form.  Because $p' \neq q'$, $R(p') \neq R(q')$ and $L(p') \neq L(q')$ by Proposition $\ref{DubRestAreBijection}$, but 
$$R(L(p')) = M(p') = M(q') = R(L(q')).$$
Thus there is an $1 \leq i \leq 2n-2$ so that $L(p')_0 < L(p')_i$ and $L(q')_0 > L(q')_i$.  If $1 \leq i \leq 2n-3$, we find a contradiction in the same fashion as in Proposition $\ref{DubRestAreBijection}$, case $\bf{(a.1)}$.  Thus we can assume that $i = 2n-2$ is the only $i$ so that $L(p')_0 < L(p')_i$ and $L(q')_0 > L(q')_i$.  Thus 
$$ L(p')_0 < L(p')_{2n-2} \implies p'_0 < p'_{2n-2} \implies p_0 < p_{2n-1} \implies L^2(p)_0 < L^2(p)_{n-1} $$
$$ L(q')_0 > L(q')_{2n-2} \implies q'_0 > q'_{2n-2} \implies q_0 > q_{2n-1} \implies L^2(q)_0 > L^2(q)_{n-1} $$
so $L^2(p) \neq L^2(q)$, and $u = v$.  Thus, by Theorem $\ref{SameFormIFFCompPair}$ and Proposition $\ref{LengthOfAlphaForAGBForN}$, $L^2(p)$ and $L^2(q)$ are a complementary pair of type $k \geq 2$.  Thus, without loss of generality, $L^2(p)_{k-1} < L^2(p)_{n-1}$ and $L^2(q)_{k-1} > L^2(q)_{n-1}$, so $p'_{2k-2} < p'_{2n-2}$ and $q'_{2k-2} > q'_{2n-2}$.  Thus $M(p')_{2k-1} < M(p')_{2n-3}$ and $M(q')_{2k-1} > M(q')_{2n-3}$, which contradicts the assumption that $M(p') = M(q')$.

Therefore if $n \neq 2^r-1$, $2^r$, or $2^r+1$ for any $r \geq 3$, $p' = q'$ if and only if $M(p') = M(q')$.

$\vspace{0.5ex}$

$\textbf{(b)}$  Let $n = 2^r-1$, $2^r$, or $2^r+1$ for some $r \geq 3$.

$\textbf{Case (b.1)}$  Suppose $p$ and $q$ have the same form.  So for each $0 \leq i \leq n$, 
$$ p_i < p_{i+1} \iff q_i < q_{i+1}. $$
So we know for each $i$, $T[a+i]$ and $T[b+i]$ both have the same $C_j$ as a prefix, so
$$ i \in U_j \iff i \in V_j $$
and so $\abs{U_j} = \abs{V_j}$ for each $j$.

If $p = q$, then $p' = q'$ and $M(p') = M(q')$, so we can say $p \neq q$.  If $n = 2^r-1$ or $2^r$ then $p' = q'$ by Lemma $\ref{TM_pISq_DiffForm}$ and $M(p') = M(q')$, so we can say $n = 2^r+1$ for some $r \geq 3$.  Thus $p$ and $q$ are a complementary pair of type 3 by Theorem $\ref{SameFormIFFCompPair}$ and Proposition $\ref{LengthOfAlphaForAGBForN}$, and $L^2(p)$ and $L^2(q)$ are a complementary pair of type 1 by Proposition $\ref{ImageOfTypeK}$.  So, without loss of generality, there is some $1 \leq x \leq n-1$ so that $L^2(p)_0 = L^2(q)_{n-1} = x$ and $L^2(p)_{n-1} = L^2(q)_0 = x+1$, and for each $1 \leq i \leq n-2$ $L^2(p)_i = L^2(q)_i$.

Since $p$ and $q$ are a complementary pair of type 3 we know $T[a,a+1] = T[a+n-1,a+n]$, thus we know $T[b,b+1] = T[b+n-1,b+n] = T[a,a+1]$ because $u=v$.  So there is a $j$ so that each of $T[a]$, $T[a+n-1]$, $T[b]$, and $T[b+n-1]$ each have $C_j$ as a prefix.  So by Proposition $\ref{CalcTheFwdImage}$, there are some $y$ and $z$ so that
\begin{align*}
p'_0 = y &\hspace{8.0ex} q'_0 = y+1 \\
p'_1 = z &\hspace{8.0ex} q'_1 = z+1 \\
p'_{2n-2} = y+1 &\hspace{8.0ex} q'_{2n-2} = y \\
p'_{2n-1} = z+1 &\hspace{8.0ex} q'_{2n-1} = z 
\end{align*}
and for each $2 \leq i \leq 2n-3$, $p'_i = q'_i$.  The order of $y$ and $z$ will be either $y < y+1 < z < z+1$ (so $T_a = T_b = 0$) or $z < z+1 < y < y+1$ (so $T_a = T_b = 1$).  If $y < y+1 < z < z+1$, then $M(p')_0 = z-1 = M(q')_0$ and $M(p')_{2n-2} = y = M(q')_{2n-2}$.  If $z < z+1 < y < y+1$, then $M(p')_0 = z = M(q')_0$ and $M(p')_{2n-2} = y-1 = M(q')_{2n-2}$.  In either case we have, for $2 \leq i \leq 2n-3$,
$$ p'_i < y \iff q'_i < y+1 \hspace{3.0ex} \text{ and } \hspace{3.0ex} p'_i < z+1 \iff q'_i < z $$
so $M(p')_{i-1} = M(q')_{i-1}$.  Therefore $M(p') = M(q')$.

Therefore if $p$ and $q$ have the same form then $M(p') = M(q')$.

$\textbf{Case (b.2)}$  Suppose $p$ and $q$ do not have the same form, and assume $M(p') = M(q')$.  If $p$ and $q$ do not have the same form, there is an $0 \leq i \leq n$ so that, without loss of generality, $p_i < p_{i+1}$ and $q_i > q_{i+1}$ and $p \neq q$.  By Lemma $\ref{TM_pISq_DiffForm}$, $p' \neq q'$.  Then as in Case $(a.2)$ we find a contradiction to the assumption, so $M(p') \neq M(q')$.

Therefore $p$ and $q$ have the same form if and only if $M(p') = M(q')$.
$\qed$
\end{proof}

Thus, for $n \geq 9$, the map $\delta_M$ 
when applied to permutations associated with the Thue-Morse word are injective when $n \neq 2^r-1$, $2^r$, or $2^r+1$ for any $r \geq 3$, so $\abs{\mathrm{Perm}^{d(T)}_{odd}(2n-2)} = \abs{\mathrm{Perm}^T(n+2)}$.

When $n = 2^r-1$, $2^r$, or $2^r+1$ for some $r \geq 3$ the map $\delta_M$ is surjective, but not injective.  In this case, if $p$ and $q$ are subpermutations of $\pi_T$ of length $n+2$, where $p$ has form $u$ and $q$ has form $v$, $\abs{u} = \abs{v} = n+1$, $\delta_M(p) = \delta_M(q)$ if and only if $u = v$.  Thus the number of subpermutations of $\pi_{d(T)}$ of length $2n-2$ which start in an odd position are determined by the number of factors of $T$ of length $n+1$, so $\abs{\mathrm{Perm}^{d(T)}_{odd}(2n-2)} = \abs{\scr{F}_T(n+1)}$.

We are now ready to calculate the permutation complexity of $d(T)$.
\begin{theorem}
\label{DubTMPermComp}
For the Thue-Morse word $T$, let $n \geq 9$.
\begin{itemize}
\item[(a)] If $n = 2^r$, then $$ \tau_{d(T)}(2n-1) = 2^{r+2} + 2^{r+1} $$ $$ \tau_{d(T)}(2n) = 2^{r+2} + 2^{r+1} + 4 $$ 
\item[(b)] If $n = 2^r + p$ for some $0 < p \leq 2^r-1$, then $$ \tau_{d(T)}(2n-1) = 2^{r+3} + 4p $$ $$ \tau_{d(T)}(2n) = 2^{r+3} + 4p + 2 $$ 
\end{itemize}
\end{theorem}
\begin{proof}
Let $n \geq 9$.  

$\vspace{1.0ex}$

$\textbf{(a)}$  Suppose $n = 2^r$.  So $2n = 2(2^r) = 2(2^{r} + 1) -2$, and from Lemma $\ref{TM_pISq_DiffForm}$ and Lemma $\ref{ThueMorseDubRestAreBijection}$ each of the maps $\delta$, $\delta_L$, $\delta_R$, and $\delta_M$
are only surjective.  

So $n + 1 = 2^r + 1 = 2^{r-1} + 2^{r-1} + 1$, and $n + 2 = 2^r + 2 = 2^r + 1 + 1$.  So by Proposition $\ref{TMFactComplexity}$
\begin{align*}
\tau_{d(T)}(2n-1) &= \abs{\mathrm{Perm}^{d(T)}_{ev}(2n-1)} + \abs{\mathrm{Perm}^{d(T)}_{odd}(2n-1)} = \abs{\scr{F}_T(n+1)} + \abs{\scr{F}_T(n+1)} \\
&= 8(2^{r-2}) + 2(2^{r-1}) + 8(2^{r-2}) + 2(2^{r-1}) = 2^{r+2} + 2^{r+1} \\
\\
\tau_{d(T)}(2n) &= \abs{\mathrm{Perm}^{d(T)}_{ev}(2n)} + \abs{\mathrm{Perm}^{d(T)}_{odd}(2n)} = \abs{\scr{F}_T(n+1)} + \abs{\scr{F}_T(n+2)} \\
&= 8(2^{r-2}) + 2(2^{r-1} ) + 6(2^{r-1}) + 4(1) = 2^{r+2} + 2^{r+1} + 4
\end{align*}

$\vspace{1.0ex}$

$\textbf{(b)}$  Suppose $n = 2^r + p$.  There will be 3 cases to consider.  First when $0 < p \leq 2^r-3$, next when $p = 2^r-2$, and finally when $p = 2^r-1$.

$\vspace{0.5ex}$

$\textbf{(b.1)}$  Suppose $0 < p \leq 2^r-3$.  So $2n = 2(2^r + p) = 2(2^{r} + p + 1) -2$, and from Lemma $\ref{TM_pISq_DiffForm}$ and Lemma $\ref{ThueMorseDubRestAreBijection}$ each of the maps $\delta$, $\delta_L$, $\delta_R$, and $\delta_M$ 
are injective.  

So $n + 2 = 2^r + p + 2$, and $n + 3 = 2^r + p + 3$.  So by Theorem $\ref{PermCompIsTheFormula}$
\begin{align*}
\tau_{d(T)}(2n-1) &= \abs{\mathrm{Perm}^{d(T)}_{ev}(2n-1)} + \abs{\mathrm{Perm}^{d(T)}_{odd}(2n-1)} = \abs{\mathrm{Perm}^T(n+2)} + \abs{\mathrm{Perm}^T(n+2)} \\
&= 2(2^r + p + 2 - 2) + 2(2^r + p + 2 - 2) = 2^{r+2} + 4p \\
\\
\tau_{d(T)}(2n) &= \abs{\mathrm{Perm}^{d(T)}_{ev}(2n)} + \abs{\mathrm{Perm}^{d(T)}_{odd}(2n)} = \abs{\mathrm{Perm}^T(n+2)} + \abs{\mathrm{Perm}^T(n+3)} \\
&= 2(2^r + p + 2 - 2) + 2(2^r + p + 3 - 2) = 2^{r+2} + 4p + 2
\end{align*}

$\vspace{0.5ex}$

$\textbf{(b.2)}$  Suppose $p = 2^r-2$, so $n = 2^r+2^r-2 = 2^{r+1}-2$.  From Lemma $\ref{TM_pISq_DiffForm}$ each of the maps $\delta$, $\delta_L$, and $\delta_R$ 
are injective.  Then we have $2n = 2(2^{r+1}-2) = 2(2^{r+1} -1) -2$ and by Lemma $\ref{ThueMorseDubRestAreBijection}$ the map $\delta_M$
is only surjective.  

So $n + 2 = 2^{r+1} = 2^{r} + 2^{r} = 2^{r} + (2^{r} - 1) + 1$.  So by Proposition $\ref{TMFactComplexity}$ and Theorem $\ref{PermCompIsTheFormula}$
\begin{align*}
\tau_{d(T)}(2n-1) &= \abs{\mathrm{Perm}^{d(T)}_{ev}(2n-1)} + \abs{\mathrm{Perm}^{d(T)}_{odd}(2n-1)} = \abs{\mathrm{Perm}^T(n+2)} + \abs{\mathrm{Perm}^T(n+2)} \\
&= 2(2^{r+1} + 2^r - 2) + 2(2^{r+1} + 2^{r} - 2) = 2^{r+3} + 2^{r+2} - 8 = 2^{r+3} + 4(2^{r} - 2) \\
\\
\tau_{d(T)}(2n) &= \abs{\mathrm{Perm}^{d(T)}_{ev}(2n)} + \abs{\mathrm{Perm}^{d(T)}_{odd}(2n)} = \abs{\mathrm{Perm}^T(n+2)} + \abs{\scr{F}_T(n+2)} \\
&= 2(2^{r+1} + 2^r - 2) + 8(2^{r-1}) + 2(2^{r} - 1) = 2^{r+3} + 2^{r+2}-6 = 2^{r+3} + 4(2^{r}-2) + 2
\end{align*}

$\vspace{0.5ex}$

$\textbf{(b.3)}$  Suppose $p = 2^r-1$, so $n = 2^r+2^r-1 = 2^{r+1}-1$.  So $2n = 2(2^{r+1}-1) = 2(2^{r+1}) -2$, and from Lemma $\ref{TM_pISq_DiffForm}$ and Lemma $\ref{ThueMorseDubRestAreBijection}$ each of the maps $\delta$, $\delta_L$, $\delta_R$, and $\delta_M$ 
are only surjective.  

So $n + 1 = 2^{r+1} = 2^{r} + (2^{r} - 1) + 1$, and $n + 2 = 2^{r+1} + 1 = 2^{r} + 2^{r} + 1$.  So by Proposition $\ref{TMFactComplexity}$
\begin{align*}
\tau_{d(T)}(2n-1) &= \abs{\mathrm{Perm}^{d(T)}_{ev}(2n-1)} + \abs{\mathrm{Perm}^{d(T)}_{odd}(2n-1)} = \abs{\scr{F}_T(n+1)} + \abs{\scr{F}_T(n+1)} \\
&= 8(2^{r-1}) + 2(2^{r} - 1) + 8(2^{r-1}) + 2(2^{r} - 1) = 2^{r+3} + 2^{r+2} - 4 = 2^{r+3} + 4(2^{r} - 1) \\
\\
\tau_{d(T)}(2n) &= \abs{\mathrm{Perm}^{d(T)}_{ev}(2n)} + \abs{\mathrm{Perm}^{d(T)}_{odd}(2n)} = \abs{\scr{F}_T(n+1)} + \abs{\scr{F}_T(n+2)} \\
&= 8(2^{r-1}) + 2(2^{r} - 1) + 8(2^{r-1}) + 2(2^{r}) = 2^{r+3} + 2^{r+2}-2 = 2^{r+3} + 4(2^{r}-1) + 2
\end{align*}
$\qed$
\end{proof}

\bibliographystyle{plain}
\bibliography{mybib}

\end{document}